\documentclass[a4paper]{article}
\usepackage{graphicx} 
\usepackage{amssymb}
\usepackage{verbatim}
\usepackage{lipsum}
\usepackage{amsfonts}
\usepackage{graphicx}
\usepackage{epstopdf}
\usepackage[ruled,vlined]{algorithm2e}
\usepackage{amsmath,amsthm}
\usepackage{verbatim}
\usepackage{graphics}
\usepackage{graphicx}
\usepackage{subfigure}
\usepackage{bm}
\usepackage{multirow,multicol}
\usepackage{float}
\usepackage{appendix}
\usepackage{chngcntr}
\usepackage{color}
\usepackage{tikz}
\usepackage{pgfplots}
\usepackage{booktabs,hyperref}
\usepackage{longtable}
\usetikzlibrary{arrows.meta}
\usetikzlibrary{shapes.geometric,patterns,hobby}
\ifpdf
\DeclareGraphicsExtensions{.eps,.pdf,.png,.jpg}
\else
\DeclareGraphicsExtensions{.eps}
\fi

\newtheorem{Theorem}{Theorem}[section]
\newtheorem{Assumption}[Theorem]{Assumption}
\newtheorem{Lemma}[Theorem]{Lemma}

\newcommand{\dx}{\,\mathrm{d} x}

\DeclareMathAlphabet{\mathsfsl}{OT1}{cmss}{m}{sl}

\pgfplotsset{compat=1.17} 
\usepackage{geometry}

\begin{document}
\title{
Maximization of Supercapacitor Storage via Topology Optimization of Electrode Structures
\thanks{This work was supported in part by the National Natural Science Foundation of China under grants (No. 12401534 and No. 12471377) and the Science and Technology Commission of Shanghai Municipality (No. 22DZ2229014)}}

\author{Jiajie Li \thanks{Department of Applied Mathematics, The Hong Kong Polytechnic University, Kowloon, Hong Kong. E-mail: jiajienumermath.li@polyu.edu.hk}
\and Xiang Ji \thanks{School of Mathematical Sciences, Shanghai Jiao Tong University, Shanghai 200240, China. E-mail: xian9ji@sjtu.edu.cn}
\and Shenggao Zhou \thanks{School of Mathematical Sciences, MOE-LSC, CMA-Shanghai, and Shanghai Center for Applied Mathematics,  Shanghai Jiao Tong University, Shanghai, China. Shanghai Zhangjiang Institute of Mathematics, Shanghai, China. E-mail: sgzhou@sjtu.edu.cn}
\and Shengfeng Zhu \thanks{Key Laboratory of Ministry of Education \& Shanghai Key Laboratory of Pure Mathematics and Mathematical Practice \& School of Mathematical Sciences, East China Normal University, Shanghai 200241, China. E-mail: sfzhu@math.ecnu.edu.cn}
}

\maketitle

\begin{abstract}
    As widely used electrochemical storage devices, supercapacitors deliver higher power density than batteries, but suffer from significantly lower energy density. In this work, we propose a topology optimization model for electrode structure to maximize energy storage in supercapacitors. The existence of minimizers to the resulting optimal control problem, which is constrained by a modified steady-state Poisson--Nernst--Planck system describing ionic electrodiffusion, has been theoretically established by using the direct method in the calculus of variation. Sensitivity analysis of the topology optimization model is performed to derive variational derivatives and corresponding adjoint equations. A gradient flow formulation discretized by a stabilized semi-implicit scheme is developed to solve the resulting topology optimization problem. Extensive numerical experiments present various porous electrode structures that own large area of electrode-electrolyte interface, demonstrating the effectiveness and robustness of the proposed topology optimization model and corresponding algorithm. 
    
\end{abstract}




\section{Introduction}
Supercapacitors, also known as electric double layer capacitors, have received significant attention due to their unique capability to bridge the performance gap between conventional capacitors and batteries~\cite{supercapacitor_science_2}. It has been widely used in electric vehicles, unmanned aerial vehicles, regenerative braking, etc.~\cite{EDLC_app,Supercapacitors_ACS_EL_3}. Such vast application landscape drives the pursuit of supercapacitors with superior performance.

During the past several decades, optimization of the shape and topology that guide the design and material distributions \cite{DelZol,SokolowskiZolesio1992} has exhibited promising applications in industry and engineering, such as solid mechanics \cite{Bendsoe}, computational fluid dynamics \cite{MP}, and photonic crystals \cite{SH}.
The optimization of the shape and topology offers a promising tool to boost the performance of electrochemical devices~\cite{Yaji2018,YoonPark,Roy2022,Simon2008,Ishizuka2020}. However, existing literature on topology optimization in electrochemistry is rare and simplifies the governing electrodiffusion system to some extent. For instance, spatial concentration variation has been neglected, resulting in a governing system that involves only the Poisson equation of chemical potential~\cite{Ishizuka2020,Onishi2019,Roy2022}.


The electrodiffusion of ions in electrolytes can be accurately characterized by the Poisson--Nernst--Planck (PNP) system or their modified versions~\cite{BazantPRE,HsiehJDE2020,KilicPRE2007II,KilicPRE2007I}.
These continuum models describe the
transport of ions under the influence of both ionic concentration gradient and an electric field, providing valuable insights into the spatial and temporal dynamics of ionic species across different environments. In this work, we employ an topology optimization approach to maximize the charge storage of supercapacitors, in which electric energy is stored in ionic electric double layers that are described by the steady-state PNP system. Since electric double layers are formed at the electrode-electrolyte interface, supercapacitors often adopt nanoporous electrode structures to maximize the charge storage via increasing the electrode-electrolyte interface area~\cite{ChanutPNAS2023}. 
Historically, the study of porous electrodes began with the transmission line model~\cite{Levie_1963,Levie_1964} based on linearized electrolyte theories. 
Due to its capability to closely fit the experimental data by adjusting the parameters, the transmission line model has gained widespread applications~\cite{Wu_2022_ChemRev}. 
However, it significantly simplifies porous electrodes by neglecting its complex microstructure, such as irregular pore geometry and spatial heterogeneity.

Recently, a shape optimization approach has been used to optimize the structure of supercapacitor electrode, aiming to maximizing charge storage subject to the steady-state PNP system and a geometry constraint \cite{LZZIJNME2025}. Since larger area of electrode-electrolyte interface often leads to a higher density of energy storage, an additional perimeter constraint is required to ensure the well-posedness of the optimization problem.  Compared to shape optimization \cite{DelZol,SokolowskiZolesio1992} by adjusting the boundary to obtain a better layout, topology optimization, which can perform both shape and topological changes of geometric structures, is able to better explore the design space and possibly capture nanoporous structure of large area of electrode-electrolyte interface. Two types of implicit interface methods, such as the level set method \cite{Osher1} and the phase field method \cite{BS, JLXZ,Shen2019}, have been widely used to track both the shape and topology evolution. It is noteworthy that in topology optimization formulation, the steady-state PNP system should be appropriately modified to reflect electrolyte accessible domain that keeps evolving during the optimization. In addition, it is often desirable to theoretically establish the existence of minimizers to the optimal control problem that describes the topology optimization subject to the modified PNP system for electrodiffusion of ions.


In this work, we investigate topology optimization of electrode structures with a phase field method for maximization of supercapacitor storage~\cite{BazantPRE,Lian_PhysRevLett_2020}. The main idea behind the phase field model of topology optimization \cite{DBH,Garcke2015,JLXZ} is to combine both the objective and the so-called Ginzburg–Landau energy to construct the total free energy. The latter can be regarded as a diffuse interface approximation of perimeter regularization \cite{Garcke2015}, ensuring the existence of a solution to the optimization problem. The objective representing the total net charge in the system is maximized subject to a geometry constraint and a modified steady-state PNP system, which is proposed to describe the electrodiffusion of ions in an evolving electrolyte-accessible domain. Existence of a minimizer to the topology optimization model is rigorously established by using the direct method in the calculus of variation. A gradient flow strategy, discretized by a stabilized semi-implicit scheme, is developed to solve the model problem, and finds a local minimizer, whose phase field function represents the optimized configuration of interest.   


The rest of the paper is organized as follows. In Section 2, we bulid a topology optimization model to maximize total net charge storage with respect to electrode configuration. The existence to a minimizer of the resulting optimal control problem is theoretically established as well. 
In Section 3, we perform sensitivity analysis of the model and derive a Gâteaux derivative of the objective with respect to the phase field function. Then, a gradient flow strategy is proposed to solve the constrained optimization problem. 
In Section 4, we present  spatial and temporal discretization of the gradient flow formulation.
In Section 5,  we perform numerical experiments concerning electrode topology optimization in both 2D and 3D spaces. A brief conclusion is drawn in Section 6. 


\section{Model setting}

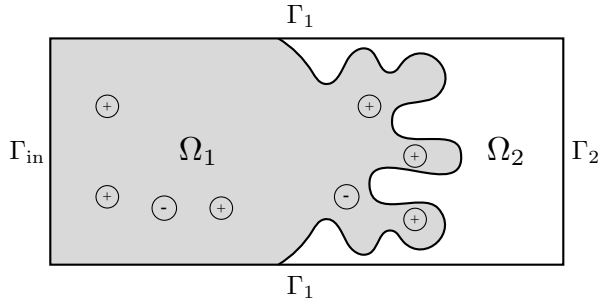
\begin{figure}[htbp]
\centering
\begin{tikzpicture}[scale=3]
	\draw[thick,fill=gray!30!white] (1.0,1.0) -- (0.0,1.0) -- (0.0,0.0) -- (1.0,0.0) to[curve through =
	{(1.15,0.15) (1.2,0.2) (1.3,0.1) (1.4,0.05) (1.5,0.15) (1.55,0.1) (1.7,0.25) (1.4,0.35)  (1.8, 0.5) (1.5,0.65) (1.7,0.75) (1.55,0.9)  (1.5,0.85) (1.4,0.95) (1.3,0.9) (1.2,0.8) (1.15,0.85)}] (1.0,1.0);
	\draw[thick] (1.0,0) -- (2.25,0) -- (2.25,1.0) -- (1.0,1.0);
	\draw (0.65,0.5) node[thick, scale=1.25] {$\Omega_1$};
	\draw (2.0,0.5) node[thick, scale=1.25] {$\Omega_2$};
	\draw (-0.1,0.5) node[thick, scale=1.] {$\Gamma_{\rm in}$};
	\draw (1.1,-0.1) node[thick, scale=1.] {$\Gamma_1$};
	\draw (1.1,1.1) node[thick, scale=1.] {$\Gamma_1$};
	\draw (2.35,0.5) node[thick, scale=1.] {$\Gamma_2$};
\draw node[circle,draw][scale=0.4] at (0.75,0.25) {\textbf{+}};
\draw node[circle,draw][scale=0.4] at (0.25,0.3){\textbf{+}};
\draw node[circle,draw][scale=0.4] at (0.25,0.7){\textbf{+}};
\draw node[circle,draw][scale=0.4] at (1.4,0.7){\textbf{+}};
\draw node[circle,draw][scale=0.4] at (1.6,0.48){\textbf{+}};
\draw node[circle,draw][scale=0.4] at (1.6,0.2){\textbf{+}};
\draw node[circle,draw][scale=0.6] at (0.5,0.25) {\textbf{-}};
\draw node[circle,draw][scale=0.6] at (1.3,0.3) {\textbf{-}};
\end{tikzpicture}
\caption{Schematic plot of a supercapacitor model on a computational domain $\Omega=\Omega_1 \cup \Omega_2$, where $\Omega_1$ shown in gray represents the electrolyte domain and $\Omega_2$ shown in white is the design electrode domain.  }
\label{fig1}
\end{figure}
We introduce a mathematical model of topology optimization subject to a geometric volume constraint to maximize ionic charge storage of supercapacitors that are described by the steady-state PNP system. Let $\Omega \subset \mathbb{R}^d\ (d= 2,3)$ be an open bounded domain with a Lipschitz boundary $\partial \Omega$ that consists of 
disjoint boundaries: $\partial\Omega = \overline{\Gamma_{\rm in}}\cup \overline{\Gamma_1}\cup \overline{\Gamma_2}$ (cf. Fig.~\ref{fig1}). Let $\psi:\Omega\rightarrow\mathbb{R}$ be the electric potential, and let $c_i:\Omega\rightarrow\mathbb{R} \ (i=1,2)$ be the concentrations of binary ionic species with valence $z_1=1$ and $z_2=-1$. 
Following the standard notations of
Sobolev spaces \cite{Adam}, we introduce a set for potential functions: $H^1_{g}(\Omega):=\{w\in H^1(\Omega): w=g\ {\rm on}\ \Gamma_{\rm in}\cup\Gamma_2 \}$, and sets for ionic concentration functions: $H^1_{c, i}(\Omega):=\{w\in H^1(\Omega): w = c_i^\infty\ {\rm on }\ \Gamma_{\rm in}\cup \Gamma_2 \}$ ($i=1,2)$, where $g\in H^2(\Omega)$ is a given boundary potential data and $c_i^\infty\in H^2(\Omega)$ are the boundary concentrations (modeling a half supercapacitor connected to an ionic reservoir at $\Gamma_{\rm in}$). Let $g\vert_{\Gamma_{\rm in}}=0$ represent zero electric potential on the left edge and $c_i^\infty\vert_{\Gamma_2}=0$ represent zero concentration on the right edge or electrode region. Moreover, let us introduce Sobolev spaces: $H^1_{d,0}(\Omega):=\{w\in H^1(\Omega): w=0\ {\rm on }\ \Gamma_{\rm in}\cup\Gamma_2 \}$, $\mathbf{H}^1(\Omega):= [H^1(\Omega)]^d$, and $\textbf{H}^1_{d,0}(\Omega):=\{\bm w\in \textbf{H}^1(\Omega): \bm w =\bm 0\ {\rm on}\ \Gamma_{\rm in}\cup \Gamma_2 \}$.

\subsection{Phase field model}
The entire domain $\Omega$ is partitioned into three disjoint subregions $\overline{\Omega}:= \overline{\Omega}_1\cup \overline{\Omega}_2 \cup \overline{\Omega}_0$, where $\Omega_1$, $\Omega_2$, and $\Omega_0$ represent the domains for electrolyte, electrode, and transition diffuse layer, respectively (cf. Fig.~\ref{fig1}). The domain is represented via a phase field function $\phi:\Omega\rightarrow \mathbb{R}$ such that
\begin{equation*}\left\{
\begin{aligned}
&\phi( x )=1, \quad &&{\rm in}\ \Omega_1,\\
&0<\phi( x )< 1,&&{\rm in}\ \Omega_0,\\
&\phi( x )=0, \quad &&{\rm in}\ \Omega_2.
\end{aligned}\right.
\end{equation*}
Let $D_0\ (\text{resp.}\ D_m)$ and $\epsilon_0\ (\text{resp.}\ \epsilon_m)$ denote the diffusion and dielectric coefficients in the electrolyte (resp. electrode) domain, with $D_0 \gg D_m >0$ and $\epsilon_m\gg\epsilon_0 >0$. Using material interpolation with a positive integer $p\in \mathbb{Z}^+$ (typically $ p\geq 2)$, the diffusion and dielectric functions can be represented by the phase field function as $$
D(\phi) = D_m +\max\{ \min\{ \phi^p,1\}, 0\}(D_0- D_m)\quad{\rm and}\quad
\epsilon(\phi)=\epsilon_m  + \max\{ \min\{ \phi^p,1\}, 0\}(\epsilon_0-\epsilon_m),$$
respectively over the entire computational domain. Consider the electrostatic free-energy functional
\begin{equation*}
\mathcal{F}(c_1,c_2):=\int_\Omega \left(\sum_{i=1}^2 z_ic_i\right)\psi+\sum_{i=1}^2 c_i(\log(c_i)-1)\dx + \int_\Omega c_i \alpha_0(1-\phi)\dx,
\end{equation*}
where \(\alpha_0\gg 1\) characterizes repulsive interactions between the electrode and ions. The variational derivative of \(\mathcal{F}\) is given by
\begin{equation*}\mu_i = \frac{\delta\mathcal{F}}{\delta c_i}=z_i \psi +\log(c_i)+ \alpha_0(1-\phi),\quad i=1,2.\end{equation*}
Then, a modified time-dependent Nernst--Planck equation is expressed as
\begin{equation*}
\frac{\partial c_i}{\partial t}= \nabla\cdot \left(D(\phi)c_i \nabla \mu_i \right)= \nabla\cdot\left[D(\phi)\nabla c_i+D(\phi)\left(z_i\nabla\psi - \alpha_0\nabla\phi \right)c_i \right],\quad i=1,2.
\end{equation*}
The electro-diffusion of ions in the entire computational domain is governed by the modified steady-state PNP system~\cite{Rubinstein_1990}
$$
- \nabla \cdot\left[ D(\phi)(\nabla c_i + z_i c_i \nabla \psi-\alpha_0c_i\nabla\phi) \right]=0\quad {\rm in}\ \Omega,\quad i=1,2,
$$
with boundary conditions $ D(\phi)(\nabla c_i + z_i c_i \nabla \psi-\alpha_0c_i\nabla\phi)\cdot \bm{n}=0$ on $\Gamma_1$ and $c_i=c_i^\infty$ on $\Gamma_{\rm in}\cup \Gamma_2$, where $\bm{n}$ is an outward unit normal to the boundary $\Gamma_1$.
The electric potential on the entire domain is described by the Poisson equation
\begin{equation*}
- \nabla\cdot \left[\epsilon(\phi) \nabla \psi\right] = \sum_{i=1}^2 z_i c_i\quad {\rm in}\ \Omega,
\end{equation*}
with boundary conditions $\epsilon(\phi)\nabla \psi\cdot \bm{n}=0$ on $\Gamma_1$ and $\psi=g$ on $\Gamma_{\rm in}\cup \Gamma_2$.
Denote by $$V(\phi) := \int_{\Omega} \max\{\min\{\phi, 1\},0\} \,\mathrm{d}{x},$$ the volume of electrolyte region and \(V_0\) the target volume under consideration. 
It is of practical interest for electrochemical energy storage devices, such as supercapacitors, to maximize the total net charge storage~\cite{BazantPRE,Lian_PhysRevLett_2020}. 

In this work, we consider to optimize an objective of the total net charge as
\begin{equation}\label{Objfun}
\mathcal{J}(\bm c(\phi)):=\int_{\Omega} j(x, \bm c(\phi))\dx, \quad\text{with}\quad j(x, \bm c(\phi)):=-\sum_{i=1}^2 z_i c_i(\phi).
\end{equation}
With the phase field description, we study the total free energy
\begin{equation}\label{TopPro}
\mathcal{W}(\phi,\bm c(\phi)):= \int_{\Omega}\bigg[\frac{\kappa}{2}\vert\nabla\phi\vert^2+\frac{1}{\kappa}\omega(\phi)\bigg]\dx + \mathcal{J}(\bm c(\phi)),
\end{equation}
where $\kappa >0$ is a diffusive parameter, $$\omega(\phi):=\frac{1}{4}\phi^2(\phi-1)^2,$$ is a double well potential, and $\bm c:=[c_1,c_2]^{\rm T}$. The first integral on the right-hand side of (\ref{TopPro}) is the Ginzburg-Landau energy~\cite{Garcke2015}, which approximates proportionally the perimeter/area of the electrolyte-electrode interface in the sense of $\Gamma$-convergence \cite{Modica1987} as $\kappa \rightarrow 0^+$. The double well potential enforces phase separation, driving \(\phi(x)\) toward 0 or 1 for each $x\in \Omega$. Introduce a set of admissible phase field functions as 
\begin{equation}
     \mathcal{A}:=\{ \phi\in W^{1,\infty} (\Omega): \phi\vert_{\Gamma_{\rm in}}=1,\phi\vert_{\Gamma_2}=0, V(\phi)=V_0\}.
\end{equation}
The topology optimization problem is characterized by the following optimal control problem:
\begin{equation}\label{TopOpt}
\min_{\phi \in \mathcal{A}}\ \mathcal{W}(\phi,\bm c(\phi)),
\end{equation}
subject to the variational form of the steady-state PNP system: Find $c_i\in H^1_{c, i}(\Omega)\ (i=1,2)$ and $\psi\in H^1_{g}(\Omega)$ such that
\begin{equation}\label{weakPNP}
\left\{
\begin{aligned}
&\int_{\Omega}D(\phi)(\nabla c_i + z_i c_i \nabla \psi-\alpha_0c_i\nabla\phi)\cdot \nabla s  \dx=0,\quad && \forall s \in H^1_{d,0}(\Omega),\\
&\int_{\Omega} \epsilon(\phi)\nabla\psi\cdot \nabla \zeta - \bigg(\sum_{i=1}^2 z_ic_i\bigg)\zeta\dx=0,&& \forall\zeta \in H^1_{d,0}(\Omega).
\end{aligned}\right.
\end{equation}
To solve the problem (\ref{TopOpt}) numerically, the classic shape optimization method using mesh moving \cite{LZZIJNME2025} can only perform boundary variations of the electrolyte region without topological changes during shape evolution, and thus design possibilities are limited. Compared with shape optimization, the phase field model can be performed on the entire fixed design domain to explore the complicated electrolyte-electrode interface, at which the formed electric double layers account for the main net charge storage. As a diffusive interface tracking technique, the phase field method arising from material sciences has shown capabilities in topology optimization for, e.g., structural design \cite{BC,BS,DBH,JLXZ} and control of fluid flows \cite{Garcke2015}.

\subsection{Existence results}
We first introduce the modified Slotboom variables 
\begin{equation}\label{SoltTrans}
\rho_i  = c_i \exp(z_i\psi-\alpha_0\phi),\quad i=1,2.
\end{equation}
Let the boundary data $\rho_i^\infty\in H^2(\Omega)$ satisfy $\rho_i^\infty\vert_{\Gamma_{\rm in}} = c_i^\infty\exp{(-\alpha_0)}$ and $\rho_i^\infty\vert_{\Gamma_2}= 0$. Correspondingly, we denote $$H^1_{\rho, i}(\Omega):=\{ w\in H^1(\Omega): w= \rho_i^\infty\ {\rm on}\ \Gamma_{\rm in}\cup \Gamma_2 \}.$$ Then the steady-state PNP system \eqref{weakPNP} is equivalent to the following variational form: Given $\phi\in \mathcal{A}$, find $\rho_i\in H^1_{\rho, i}(\Omega)$ and $\psi \in H^1_g(\Omega)$ such that
\begin{equation}\label{PNPslot}\left\{
\begin{aligned}
&\int_{\Omega} D(\phi) \exp{(\alpha_0\phi-z_i\psi)}\nabla \rho_{i}\cdot \nabla s\dx =0,\quad&& \forall s \in H^1_{d,0}(\Omega) \\
&\int_{\Omega} \epsilon(\phi) \nabla \psi \cdot \nabla \zeta \dx - \int_{\Omega} \bigg(\sum_{i=1}^2 z_i \rho_{i} \exp{(\alpha_0\phi- z_i \psi)}\bigg) \zeta \dx=0\quad && \forall \zeta\in H^1_{d,0}(\Omega).
\end{aligned}\right.
\end{equation}

To establish the existence and uniqueness of the solution to \eqref{PNPslot}, we state assumptions on the geometry, boundary data, and model parameters (see \cite[pages 34, 41, and 44]{Markowich1986}).

\begin{Assumption}\label{assumptionPNP}

\textbf{(H1)}: Let $\Omega$ be a bounded domain of class $C^{0,1}$ in $\mathbb{R}^d$. Let $\Gamma_1$ consist of $C^2$-segments, and the $(d-1)$-dimensional Lebesgue measure of $\Gamma_{\rm in}\cup \Gamma_2$ is positive.

\textbf{(H2)}: The Dirichlet boundary data satisfy
\begin{equation*}
\begin{aligned}
 &(\rho_1^{\infty}, \rho_2^{\infty}, g)\in \big[H^{2}(\Omega)\big]^3,  (\rho_1^{\infty}, \rho_2^{\infty}, g)\big\vert_{\Gamma_{\rm in}\cup\Gamma_2} \in \big[L^{\infty}(\Gamma_{\rm in}\cup\Gamma_2)\big]^3, \frac{\partial \rho_1^\infty}{\partial \bm n}\bigg\vert_{\Gamma_1} = \frac{\partial \rho_2^\infty}{\partial \bm n}\bigg\vert_{\Gamma_1}= \frac{\partial g}{\partial \bm n}\bigg\vert_{\Gamma_1}=0,
\end{aligned}
\end{equation*}
and there exists $U_+ \geq 0$ such that
\begin{equation*}
e^{-U_+}\leq \inf_{ x  \in\Gamma_{\rm in}\cup\Gamma_2} \rho_1^{\infty}, \inf_{ x  \in \Gamma_{\rm in}\cup\Gamma_2} \rho_2^{\infty};\quad \sup_{ x  \in\Gamma_{\rm in}\cup\Gamma_2} \rho_1^{\infty}, \sup_{ x  \in\Gamma_{\rm in}\cup\Gamma_2} \rho_2^{\infty}\leq e^{U_+}.
\end{equation*}

\textbf{(H3)}: For each  $\phi \in \mathcal{A}$, the diffusivity function $D(\phi)\in W^{1,\infty}(\Omega)$ and dielectric coefficient $\epsilon(\phi)\in W^{1,\infty}(\Omega)$ satisfy $0<D_m \leq D(\phi)( x)\leq D_0$ and $0<\epsilon_0 \leq \epsilon(\phi)( x)\leq \epsilon_m$ a.e. in $\Omega$, respectively.

\textbf{(H4)}: The solution $w$ of $\Delta w = f$ in $\Omega$ with $w\vert_{\Gamma_{\rm in}\cup\Gamma_2}=0$ and $\frac{\partial w}{\partial\bm n}\vert_{\Gamma_1}=0$ satisfies $$\| w\|_{W^{2,q}(\Omega)}\leq \mathcal{C}\| f\|_{L^q(\Omega)},$$ for every $f\in L^q(\Omega)$ with $q=2$ or $q=\frac{3}{2}$, and some generic constant $\mathcal{C}>0$.

\end{Assumption}

\begin{Lemma}\label{PNPLinf}
Let Assumption \ref{assumptionPNP} hold, and suppose the boundary data $g$ is sufficiently small. Given $\phi\in \mathcal{A}$, the problem \eqref{PNPslot} admits a unique solution $(\rho_1,\rho_2,\psi)\in \big[H^2(\Omega)\cap L^{\infty}(\Omega)\big]^3$, satisfying the $L^{\infty}$-estimates
\begin{equation*}
\begin{aligned}
    &\exp{(-U_+)}\leq \rho_i(x)\leq \exp{(U_+)},\quad {\rm a.e.}\ {\rm in}\ \Omega,\quad i=1,2, \\  
    &\min\bigg( \inf_{ x \in \Gamma_{\rm in}\cup \Gamma_2}g( x ), -U_+ \bigg)\leq \psi(  x )\leq \max\bigg(\sup_{ x  \in\Gamma_{\rm in}\cup \Gamma_2}g(  x ), U_+ \bigg),\quad {\rm a.e.}\ {\rm in}\ \Omega,
\end{aligned}
\end{equation*}
\end{Lemma}
where $U_+>0$ is a constant independent of $\phi$.  
\begin{proof}
The existence proof proceeds by decoupling the steady-state PNP system \eqref{PNPslot}, following the framework presented in \cite[pages 33--45]{Markowich1986}. The key difference here lies in the variable coefficients \(D(\phi)\), $\exp{(\alpha_0\phi)}$ and \(\epsilon(\phi)\) in \eqref{PNPslot}, which satisfy $\exp{(\alpha_0\phi)}$, \(D(\phi), \epsilon(\phi) \in W^{1,\infty}(\Omega)\). Since \(\epsilon(\phi) \geq \epsilon_0 > 0\) and \(D(\phi) \geq D_m > 0\) for all \({x} \in \Omega\), the Poisson--Boltzmann type of equation in \eqref{PNPslot} is \(H_0^1\)-elliptic (i.e., coercive), with the fixed $\rho_i$. Following \cite[Theorem 3.2.1]{Markowich1986}, the coercivity and regularity of coefficients yield the existence and \(L^{\infty}\)-boundedness of solutions independent of $\phi$; cf.~\cite{PBE_SIAP11, PBE_arX24} as well.  The unique solvability and pointwise boundedness for $\rho_i$ of \eqref{PNPslot} independent of $\phi$ is an immediate result according to \cite[Lemma 3.2.2]{Markowich1986}. Then the existence of a solution of the coupled system \eqref{PNPslot} follows by showing the completely continuous mapping and utilizing Leray-Schauder's Theorem; cf.~\cite[Lemmas 3.2.3 - 3.2.5]{Markowich1986}. The \(H^2\)-regularity stems from \(L^q\)-theory for elliptic equations (\(q > 1\); see \cite[Theorem 3.3.1]{Markowich1986}) and the assumptions on boundary conditions in \textbf{(H2)}. If the solution is a local branch of the isolated equilibrium solution, the uniqueness of the weak solution holds under the smallness condition on $g$ (cf. \cite[Theorem 3.4.1]{Markowich1986}).  
\end{proof}

\begin{Assumption}\cite[page 234]{Garcke2015}\label{AssumptionJ}
Let $j( x , \bm z): \Omega\times \mathbb{R}^2\rightarrow \mathbb{R}$ be a Carath\'{e}odory function fulfilling
\begin{itemize}
    \item $j(\cdot, \bm z): \Omega\rightarrow \mathbb{R}$ is measurable for each $\bm z \in \mathbb{R}^2$.

    \item $j(x ,\cdot): \mathbb{R}^2\rightarrow \mathbb{R}$ is continuous for almost every $ x \in\Omega$.
\end{itemize}
Assume there exist non-negative functions $a\in L^1(\Omega)$ and $b\in L^{\infty}(\Omega)$ such that for almost every $ x \in \Omega$
\begin{equation}\label{upCon}
\vert j( x , \bm z)\vert \leq a( x ) + b(x )| \bm z|^2,
\end{equation}
for all $\bm z \in \mathbb{R}^2$. Additionally, $\mathcal{J}(\bm c): H^1_{c, 1}(\Omega)\times H^1_{c, 2}(\Omega)\rightarrow \mathbb{R}$ is weakly lower semicontinuous and bounded from below.
\end{Assumption}

We define $$\underline{\psi}:=\min\big( \inf_{ x \in \Gamma_{\rm in}\cup \Gamma_2}g( x ), -U_+ \big)\ \ {\rm and}\ \
\overline{\psi}:=\max\big( \sup_{ x \in \Gamma_{\rm in}\cup \Gamma_2}g( x ), U_+ \big),$$
and
$$\underline{\mathcal{E}}_i:=\min\big\{\exp{(-z_i \overline{\psi})},\exp{(-z_i \underline{\psi})}\big\}\quad {\rm and}\quad \overline{\mathcal{E}}_i:=\max\big\{\exp{(-z_i \overline{\psi})},\exp{(-z_i \underline{\psi})}\big\}.$$

We show the existence of the solution to the optimal control problem \eqref{TopOpt} under sufficient regularity. Specifically, the coercivity of the total free energy \(\mathcal{W}\), combined with the compactness properties of Sobolev spaces, allows the application of the direct method in the calculus of variations.

\begin{Theorem}
Assume that the conditions of Lemma \ref{PNPLinf} and Assumption \ref{AssumptionJ} are satisfied. Then, there exists at least one minimizer $(\phi^*,c_1^*,c_2^*,\psi^*)\in \mathcal{A}\times H_{c,1}^1(\Omega)\times H_{c,2}^1(\Omega)\times H^1_g(\Omega) $ for the optimal control problem \eqref{TopOpt} subject to \eqref{weakPNP}.
\end{Theorem} 
\begin{proof}
We apply the direct method in the calculus of variations to establish the assertion, and divide the proof into following 3 steps.

\emph{Step 1 (Extract minimizing sequences and their limits)}: Since the Ginzburg-Landau energy in \(\mathcal{W}\) is nonnegative and Assumption \ref{AssumptionJ} implies lower boundedness of $\mathcal{J}(\bm c)$, the total energy $\mathcal{W}(\phi, \bm c): \mathcal{A}\times H^1_{c, 1}(\Omega)\times H^1_{c, 2}(\Omega)\rightarrow \mathbb{R}$ is uniformly bounded below. Thus, we select a minimizing sequence $(\phi_n, \bm c_n, \psi_n)_{n\in \mathbb{N}} \in \mathcal{A}\times \left( H^1_{c, 1}(\Omega)\times H^1_{c, 2}(\Omega)\right) \times H_g^1(\Omega)$ such that
\begin{equation*}
\lim_{n\rightarrow \infty} \mathcal{W}(\phi_n, \bm c_n) = \underbrace{\inf}_{\phi\in \mathcal{A},\, (\bm c,\psi)\ {\rm satisfy}\ \eqref{weakPNP}} \mathcal{W}(\phi, \bm c)\geq -\mathcal{C}_0 >- \infty,
\end{equation*}
where \((\bm{c}_n, \psi_n)\) is the unique solution of \eqref{weakPNP} corresponding to \(\phi_n\) and $\bm c_n:= [c_{1,n}, c_{2,n}]^{\rm T}$. For any $\varepsilon_0 >0$, there exists an $N\in\mathbb{N}$ such that for any integer $n>N$
\begin{equation*}
-\mathcal{C}_0+\frac{\kappa}{2}\| \nabla \phi_n\|_{L^2(\Omega)}^2\leq \mathcal{W}(\phi_n, \bm c_n)\leq \underbrace{\inf}_{\phi\in \mathcal{A}, (\bm c,\psi)\ {\rm satisfy}\ \eqref{weakPNP}} \mathcal{W}(\phi, \bm c)+\varepsilon_0.
\end{equation*}
Thus, the sequence $\{\nabla \phi_n \}_{n> N}$ is uniformly bounded in $L^2 (\Omega)$. By the compact embedding Theorem $H^1(\Omega)\hookrightarrow L^2(\Omega)$ and the closedness of $\mathcal{A}$, the Banach-Alaoglu Theorem guarantees the existence of a subsequence (still denoted \(\{\phi_n\}_{n \in \mathbb{N}}\)) and an element \(\phi_* \in \mathcal{A}\) such that
\begin{equation*}
\phi_n\rightarrow\phi_*\ \text{in}\ L^2(\Omega),\ \phi_n\rightarrow\phi_*\ \text{in}\ L^1(\Omega),\ \phi_n \rightarrow \phi_*\ \text{a.e. in}\ \Omega,
\end{equation*}
where the a.e. convergence follows.
By the Slotboom transformation $c_{i,n} = \rho_{i,n} \exp(-z_i\psi_n-\alpha_0 \phi)$, we shall show that $\{\rho_{i,n}\}_{n\in\mathbb{N}}$ and $\{\psi_n\}_{n\in\mathbb{N}}$ are \emph{uniformly} bounded in $H^1(\Omega)$. 
By the continuous dependence of the solutions on the data (cf.~\cite[Lemma 3.2.3]{Markowich1986}), we have the estimates
\begin{equation}\label{rhoH1}
 \| \rho_{i,n}\|_{H^1(\Omega)} 
\leq   \mathcal{C}_i \| c_i^{\infty}\|_{H^{\frac{1}{2}}(\Gamma_{\rm in}\cup \Gamma_2)},
\end{equation}
where $\mathcal{C}_i$ $ (i=1,2)$ only depends on $\Omega$, $\alpha_0$, $D_0$, $D_m$, $U_+$, and $g$. Hence, the sequence $\{\rho_{i,n}\}_{n\in\mathbb{N}}$ is uniformly bounded in $H^1(\Omega)$. On the other hand, it follows from Lemma~\ref{PNPLinf} that the unique existence of $\psi_n$ holds. With the given $\psi_n$ and $L^\infty$ estimate in Lemma \ref{PNPLinf}, we treat the second equation of \eqref{PNPslot} as a Poisson equation and obtain by the continuous dependence of the solution on the right-hand side of the equation that
\begin{equation}\label{fowiej}
\begin{aligned}
\| \psi_n\|_{H^1(\Omega)}\leq \mathcal{C}_3( \exp{(-\underline{\psi})}\| \rho_{1,n}  \|_{L^2(\Omega)}+\exp{(\overline{\psi})}\| \rho_{2,n}  \|_{L^2(\Omega)}+\| g\|_{H^{\frac{1}{2}}(\Gamma_{\rm in}\cup \Gamma_2)}) < \infty,
\end{aligned}
\end{equation}
where $\mathcal{C}_3>0$ depends only on $\Omega$.

By the Banach-Alaoglu theorem, there exist subsequences (relabelled) such that
$$\rho_{i,n}\rightharpoonup \rho_i^*\ {\rm weakly}\  {\rm in}\ H^1(\Omega)\quad {\rm and}\quad\psi_n \rightharpoonup \psi^*\ {\rm weakly\ in}\ H^1(\Omega)\quad {\rm as}\ n\rightarrow\infty,$$ where $\rho_i^*\in H_{\rho,i}^1(\Omega)$ and $\psi^*\in H^1_g(\Omega)$. Next, we verify that the limit $(\rho_1^*,\rho_2^*,\psi^*)$ is indeed a solution of \eqref{PNPslot} for the prescribed $\phi_*\in \mathcal{A}$. By the compactness embedding $H^1(\Omega)\hookrightarrow L^4(\Omega)$, we further obtain convergence of subsequences: $$\rho_{i,n}\rightarrow \rho_i^*\ {\rm in} \ L^2(\Omega), \quad \psi_n \rightarrow \psi^* \ {\rm in}\ L^4(\Omega)\quad {\rm as}\ n\rightarrow\infty.$$

\emph{Step 2 (Verify the limits satisfying the PNP system)}: By the Cauchy-Schwarz inequality, subtracting the Poisson--Boltzmann type of equation in \eqref{PNPslot} for \(\phi_n\) and \(\phi^*\) yields
\begin{equation}\label{sss0}
\begin{aligned}
&\left|\int_{\Omega} \epsilon(\phi_n) \nabla \psi_n \cdot \nabla \zeta-\epsilon(\phi^*) \nabla \psi^* \cdot \nabla \zeta\dx\right| \\
\leq & \left|\int_{\Omega} (\epsilon(\phi_n)-\epsilon(\phi^*)) \nabla \psi_n \cdot \nabla \zeta\dx\right|+ \left|\int_{\Omega}\epsilon(\phi^*) (\nabla \psi_n - \nabla \psi^*)\cdot \nabla \zeta\dx\right|,\quad\forall \zeta\in H^1_{d,0}(\Omega).
\end{aligned}
\end{equation}
 For the first term on the right-hand side of \eqref{sss0}, by the Cauchy-Schwarz inequality, we have 
\begin{equation}\label{sss01}
 \left|\int_{\Omega} (\epsilon(\phi_n)-\epsilon(\phi^*)) \nabla \psi_n \cdot \nabla \zeta\dx\right|\leq  \| (\epsilon(\phi_n)-\epsilon(\phi^*))\nabla \zeta \|_{L^2(\Omega)}\| \nabla\psi_n \|_{L^2(\Omega)}.
\end{equation}
Since $\epsilon(\cdot)$ is continuous, the integrand $(\epsilon(\phi_n)-\epsilon(\phi^*))^2\vert\nabla \zeta\vert^2$ converges to $0$ a.e. in $\Omega$ and
\begin{equation*}
 \vert \epsilon(\phi_n)-\epsilon(\phi^*) \vert^2\vert \nabla \zeta\vert^2 \leq 4\left\vert \epsilon_m+\epsilon_0\right\vert^2 \vert \nabla \zeta\vert^2\in L^1(\Omega).   
\end{equation*}
 By the Lebesgue convergence Theorem, we obtain that the right-hand side of \eqref{sss01} tends to zero as $n\rightarrow \infty$, since the sequence $\{\nabla\psi_n\}_{n\in \mathbb{N}}$ is uniformly bounded in $L^2(\Omega)$ by \eqref{fowiej}. Moreover, given that \(\psi_n \rightharpoonup \psi^*\) weakly in \(H^1(\Omega)\), the right-hand side of \eqref{sss0} also tends to zero as \(n \to \infty\).

On the other hand, for all $\zeta\in H^1_{d,0}(\Omega)$, using the Cauchy-Schwarz inequality, we have
\begin{equation}\label{sss2}
\begin{aligned}
&\left|\int_{\Omega} \exp{(\alpha_0\phi_n)} \big(\sum_{i=1}^2 z_i \rho_{i,n} \exp{(- z_i \psi_n)}\big)\zeta - \exp{(\alpha_0\phi^*)}\big(\sum_{i=1}^2 z_i \rho_{i}^* \exp{(- z_i \psi^*)}\big) \zeta \dx\right|\\
\leq &\left|\int_{\Omega} \exp{(\alpha_0\phi_n)} \big(\sum_{i=1}^2 z_i \rho_{i,n} \exp{(- z_i \psi_n)}\big)\zeta - \exp{(\alpha_0\phi^*)}\big(\sum_{i=1}^2 z_i \rho_{i,n} \exp{(- z_i \psi_n)}\big) \zeta \dx\right| \\
&+\left|\int_{\Omega} \exp{(\alpha_0\phi^*)} \big(\sum_{i=1}^2 z_i \rho_{i,n} \exp{(- z_i \psi_n)}\big)\zeta - \exp{(\alpha_0\phi^*)}\big(\sum_{i=1}^2 z_i \rho_{i}^* \exp{(- z_i \psi^*)}\big) \zeta \dx\right|\\
=&I_1+I_2.
\end{aligned}
\end{equation}
The integrand of $I_1$ has a dominated function $2\vert \exp{(\alpha_0)}-\exp{(-\alpha_0)}\vert\exp{(U_+ + \max\{\overline{\psi}, -\underline{\psi}\})}\zeta\in L^1(\Omega)$ due to the $L^\infty$ estimates of $\rho_{i,n}$ and $\psi_n$. Since the integrand of $I_1$ converges to $0$ a.e. in $\Omega$, the Lebesgue convergence Theorem implies $I_1\rightarrow 0$ as $n\rightarrow \infty$. For the rest one, we similarly have
\begin{equation}\label{sss22}
\begin{aligned}
I_2\leq&\mathcal{C}_4\left|\int_{\Omega}\big(\sum_{i=1}^2 z_i (\rho_{i,n}-\rho_i^*) \exp{(- z_i \psi_n)}\big)\zeta\dx\right| \\
&+ \mathcal{C}_4\left|\int_{\Omega} \big(\sum_{i=1}^2 z_i \rho_{i}^* [\exp{(- z_i \psi_n)}-\exp{(- z_i \psi^*)}]\big) \zeta \dx\right|\\
\leq& \mathcal{C}_4\sum_{i=1}^2 \|\rho_{i,n}-\rho_i^*\|_{L^2(\Omega)}\|\exp(-z_i \psi_n)\zeta\|_{L^2(\Omega)}\\
&+\mathcal{C}_4\sum_{i=1}^2 \| \rho_i^* \zeta\|_{L^2(\Omega)}\|\exp{(- z_i \psi_n)}-\exp{(- z_i \psi^*)}\|_{L^2(\Omega)},
\end{aligned}
\end{equation}
where $\mathcal{C}_4>0$ depends only on $\Omega, \alpha_0$. The first term on the right-hand side of \eqref{sss22} goes to zero since $\|\exp(-z_i \psi_n)\zeta\|_{L^2(\Omega)}\leq \overline{\mathcal{E}}_i \|\zeta\|_{L^2(\Omega)}$. For the second term, using the Hölder inequality and the compact embedding $H^1(\Omega)\hookrightarrow L^4(\Omega)$, we deduce $$\| \rho_i^* \zeta\|_{L^2(\Omega)}\leq \|\rho_i^*\|_{L^4(\Omega)}\|\zeta\|_{L^4(\Omega)}\leq \|\rho_i^*\|_{H^1(\Omega)}\|\zeta\|_{H^1(\Omega)}.$$ Furthermore, using the convexity and monotonicity of the exponential function, we obtain
\begin{equation}
\begin{aligned}
\|\exp{(- z_i \psi_n)}-\exp{(- z_i \psi^*)}\|_{L^2(\Omega)}\leq \overline{\mathcal{E}}_i\|\psi^* - \psi_n  \|_{L^2(\Omega)}.
\end{aligned}
\end{equation}
Passing the limit $n\rightarrow \infty$, the right-hand side of \eqref{sss2} tends to zero.

Subtracting the continuity equations in \eqref{PNPslot} corresponding to \(\phi_n\) and \(\phi^*\), we obtain
\begin{equation}\label{sss1}
\begin{aligned}
&\left\vert\int_\Omega D(\phi_n)\exp{(\alpha_0 \phi_n)} \exp{(-z_i\psi_n)} \nabla \rho_{i,n}\cdot \nabla s- D(\phi^*)\exp{(\alpha_0 \phi^*)} \exp{(-z_i\psi^*)} \nabla \rho_{i}^*\cdot \nabla s\dx\right\vert \\
&\leq \left\vert\int_\Omega [D(\phi_n)\exp{(\alpha_0 \phi_n)}-D(\phi^*)\exp{(\alpha_0 \phi^*)}] \exp{(-z_i\psi_n)} \nabla \rho_{i,n}\cdot \nabla s \dx\right\vert \\
& \quad+ \left\vert\int_\Omega D(\phi^*)\exp{(\alpha_0 \phi^*)} [\exp{(-z_i\psi_n)}-\exp{(-z_i\psi^*)}] \nabla \rho_{i,n}\cdot \nabla s \dx\right\vert\\
& \quad+ \left\vert\int_\Omega D(\phi^*)\exp{(\alpha_0 \phi^*)}\exp{(-z_i\psi^*)} (\nabla \rho_{i,n}-\nabla \rho_{i}^*)\cdot \nabla s \dx\right\vert,\quad \forall s \in H^1_{d,0}(\Omega)\\
&= I_4+ I_5+ I_6.
\end{aligned}
\end{equation}
For the first term on the right-hand side of \eqref{sss1}, by the Cauchy-Schwarz inequality, we have
\begin{equation}\label{sss3}
\begin{aligned}
I_4 \leq \overline{\mathcal{E}}_i\|(D(\phi_n)\exp{(\alpha_0 \phi_n)}-D(\phi^*)\exp{(\alpha_0 \phi^*)})\nabla s\|_{L^2(\Omega)}\| \nabla \rho_{i,n}\|_{L^2(\Omega)}.
\end{aligned}
\end{equation}
Since $D(\cdot)$ is continuous, the integrand $(D(\phi_n)\exp{(\alpha_0\phi_n)}-D(\phi^*)\exp{(\alpha_0\phi^*)})^2\vert\nabla s\vert^2$ converges to $0$ a.e. in $\Omega$. In addition, 
\begin{equation*}
 \vert D(\phi_n)\exp{(\alpha_0\phi_n)}-D(\phi^*)\exp{(\alpha_0\phi^*)} \vert^2\vert \nabla s\vert^2 \leq 4\left\vert D_m+ D_0\right\vert^2\exp{(2\alpha_0)} \vert \nabla s\vert^2\in L^1(\Omega).   
\end{equation*}
By the Lebesgue convergence Theorem, we obtain that the right-hand side of \eqref{sss3} tends to zero as $n\rightarrow \infty$, as $\rho_{i,n}$ are uniformly bounded in $H^1(\Omega)$ according to \eqref{rhoH1}. For the second term on the right-hand side of \eqref{sss1}, we estimate it as
\begin{equation}
\begin{aligned}
I_5&\leq  D_0\exp{(\alpha_0)} \Vert \exp{(-z_i\psi_n)}-\exp{(-z_i\psi^*)}\Vert_{L^4(\Omega)}\| \nabla \rho_{i,n}\|_{L^4(\Omega)}\| \nabla s\|_{L^2(\Omega)} \\
&\leq  D_0\overline{\mathcal{E}}_i \exp{(\alpha_0)}\Vert\psi_n-\psi^*\Vert_{L^4(\Omega)}\| \nabla \rho_{i,n}\|_{H^1(\Omega)}\| \nabla s\|_{L^2(\Omega)}, 
\end{aligned}
\end{equation}
where the compact embedding $H^1(\Omega)\hookrightarrow L^4(\Omega)$ is applied. Passing $n\rightarrow \infty$, $I_5$ and $I_6$ tend to zero due to the fact that $\rho_{i,n}\in H^2(\Omega)$ and $\rho_{i,n}\rightharpoonup \rho_i^*\ {\rm weakly}\  {\rm in}\ H^1(\Omega)$. Therefore, for each solution $(\rho_{i,n},\psi_n)$ solving \eqref{PNPslot} corresponding to \(\phi_n\), taking the limit $n\rightarrow \infty$ yields that the limit functions $(\rho_{i}^*,\psi^*)$ solve \eqref{PNPslot} as well.  Moreover, by Lemma \ref{PNPLinf}, it follows from $(\rho_{1,n},\rho_{2,n},\psi_n)\in \big[H^2(\Omega)\big]^3$ and $(\rho_{1,n},\rho_{2,n},\psi_{n}) \to (\rho_1^*,\rho_2^*,\psi^*)$ in $L^2(\Omega)$ that $(\rho_1^*,\rho_2^*,\psi^*)$ satisfy the corresponding boundary conditions as well. 
According to the Lemma~\ref{PNPLinf}, the uniqueness of the solution to the problem \eqref{PNPslot} guarantees that $(\rho_1^*, \rho_2^*, \psi^*)$ solves the problem given $\phi^*\in \mathcal{A}$.  

According to the Slotboom transformation \eqref{SoltTrans}, $c_i^*=\rho_i^* \exp(-z_i \psi^*+\alpha_0\phi^*)$ satisfy
\begin{equation*}
\begin{aligned}
\| c_{i,n}-c_i^*\|_{L^2(\Omega)}\leq& \overline{\mathcal{E}}_i\exp{(\alpha_0)} \|\rho_{i,n}-\rho_i^*\|_{L^2(\Omega)}+\exp(U_+)\exp{(\alpha_0)}\overline{\mathcal{E}}_i\| \psi_n-\psi^*\|_{L^2(\Omega)}\\
&+\exp{(U_+ + \overline{\mathcal{E}}_i)}\| \exp{(\alpha_0 \phi^*)}- \exp{(\alpha_0 \phi^*)}\|_{L^2 (\Omega)}\rightarrow 0\quad{\rm as}\ n\rightarrow\infty.
\end{aligned}
\end{equation*}
Therefore, the limit functions $(\phi^*, c_{i}^*,\psi^*)$ are the solutions of \eqref{weakPNP}.

\emph{Step 3 (Show minimizer via semicontinuity)}: Since $\sup_{n\in \mathbb{N}}\| \omega(\phi_n)\|_{L^{\infty}(\Omega)}<\infty$, it follows from the Lebesgue's dominated convergence theorem \cite{Garcke2015} that $$\lim_{n\rightarrow\infty}\int_{\Omega}\omega(\phi_n)\dx=\int_{\Omega}\omega(\phi^*)\dx.$$ The convexity of the gradient term leads to the weak lower semicontinuity of the Ginzburg-Landau energy functional:
\begin{equation*}
\int_{\Omega} \frac{\kappa}{2}\vert \nabla \phi^*\vert^2+\frac{1}{\kappa}\omega(\phi^*)\dx\leq \lim \inf_{\!\!\!\!\!\!\!\!n\rightarrow \infty} \int_{\Omega} \frac{\kappa}{2}\vert \nabla \phi^n\vert^2+\frac{1}{\kappa}\omega(\phi^n)\dx.
\end{equation*}
Finally, by the linearity of $\mathcal{J}(\bm c)$ with $\bm c^*:=[c_1^*, c_2^*]^{\rm T}$, we obtain:
\begin{equation*}
\mathcal{W}(\phi^*,\bm c^*)\leq \lim \inf_{\!\!\!\!\!\!\!\!n\rightarrow \infty}  \mathcal{W}(\phi_n, \bm c_n)\leq \underbrace{\inf}_{\phi\in \mathcal{A}, (\bm c,\psi)\ {\rm satisfy}\ \eqref{weakPNP}} \mathcal{W}(\phi, \bm c),
\end{equation*}
which proves that $(\phi^*,\bm c^*)$ is a minimizer of \eqref{TopOpt} subject to \eqref{weakPNP}.
\end{proof}
\section{Optimization method}
In topology optimization,  besides geometric constraints, the problem generally involves physical state constraints.  Treating the state system as an equality constraint and applying the Lagrange multiplier method typically lead to the introduction of dual variables \cite[Chapter 10]{DelZol} that are solutions of an adjoint system. Then sensitivity analysis is performed via the Gâteaux derivative of the objective w.r.t. the phase field function (cf.  \cite[Appendix A]{JCP2010PF} for sensitivity analysis in structural optimization).  
\subsection{Sensitivity analysis}
We first derive the adjoint system and then show the variational derivative of the objective.
\begin{Lemma}
Let $(c_1, c_2, \psi)\in H_{c,1}^1(\Omega)\times H_{c,2}^1(\Omega)\times H^1_g(\Omega)$ be the solution of the modified steady-state PNP system (\ref{weakPNP}). The adjoint system for the optimization problem \eqref{TopOpt} subject to \eqref{weakPNP} reads: Find $(\bm s, \zeta)\in \textbf{H}^1_{d,0}(\Omega)\times H^1_{d,0}(\Omega)$ such that
\begin{equation}\label{adjweak}
\left\{
\begin{aligned}
&\int_{\Omega} D(\phi) \nabla s_i \cdot (\nabla q + z_i q \nabla\psi - \alpha_0 q \nabla \phi)+z_i q -z_i\zeta q \dx = 0,\ i=1,2\ &&\forall q\in H^1_{d,0}(\Omega), \\
& \int_{\Omega} \epsilon(\phi) \nabla \zeta\cdot \nabla \eta \dx + \sum_{i=1}^2 \int_{\Omega} D(\phi)z_i c_i \nabla s_i \cdot \nabla \eta \dx =0 &&\forall \eta \in H^1_{d,0}(\Omega),
\end{aligned}\right.
\end{equation}
where the vectorial function $\bm s=[s_1,s_2]^{\rm T}$.
\end{Lemma}
\begin{proof}
Introduce a Lagrangian combining both the objective \eqref{Objfun} and PDE constraints \eqref{weakPNP}:
\begin{equation*}
\begin{aligned}
\mathcal{L}(\phi, \bm w, \bm v, \tau, \xi)
=&-\sum_{i=1}^2\int_{\Omega} z_i w_i \dx-\sum_{i=1}^2\int_{\Omega}D(\phi)(\nabla w_i + z_i w_i \nabla \tau- \alpha_0 w_i \nabla\phi)\cdot \nabla v_i \dx \\
&-\int_{\Omega} \epsilon(\phi)\nabla\tau\cdot \nabla \xi- \big(\sum_{i=1}^2 z_i w_i\big)\xi\dx,
\end{aligned}
\end{equation*}
where $w_i\in {H}^1_{c,i}(\Omega)\ (i=1,2),\tau\in H^1_{g}(\Omega)$, and $(\bm v, \xi)\in \textbf{H}^1_{d,0}(\Omega)\times H^1_{d,0}(\Omega)$ with $\bm v= [v_1,v_2]^{\rm T}$ and $\bm w=[w_1, w_2]^{\rm T}$. A saddle point of $\mathcal{L}$ is characterized by
\begin{equation}\label{KKTcon1}
\frac{\partial \mathcal{L}}{\partial \bm w}(\delta_{\bm w})=\frac{\partial \mathcal{L}}{\partial \tau}(\delta_{\tau})=0 \quad \forall \delta_{\bm w}=[\delta_{w_1},\delta_{w_2}]^{\rm T} \in \textbf{H}^1_{d,0}(\Omega), \delta_{\tau}\in H^1_{d,0}(\Omega).
\end{equation}
Eq. \eqref{KKTcon1} corresponds to the adjoint state system \eqref{adjweak} as derived below. The Gâteaux derivative of $\mathcal{L}$ w.r.t. $w_i$ yields the first-order necessary optimality condition
\begin{equation*}
\begin{aligned}
0 =& -\int_{\Omega} z_i \delta_{w_i}\dx - \int_{\Omega} D(\phi)(\nabla \delta_{w_i}+z_i \delta_{w_i} \nabla\tau - \alpha_0 \delta_{w_i}\nabla \phi)\cdot\nabla v_i\dx +\int_{\Omega} z_i \delta_{w_i}  \xi \dx.
\end{aligned}
\end{equation*}
Similarly, the Gâteaux derivative of the Lagrangian $\mathcal{L}$ w.r.t. $\tau$ leads to the optimality condition
\begin{equation*}
0 = -\sum_{i=1}^2 \int_{\Omega} D(\phi)z_i w_i \nabla \delta_{\tau} \cdot \nabla v_i \dx - \int_{\Omega} \epsilon(\phi) \nabla \delta_{\tau} \cdot \nabla \xi\dx,
\end{equation*}
where $(w_1, w_2, \tau)$ corresponds to the solution $(c_1, c_2, \psi)$ of the PNP system (\ref{weakPNP}) at the saddle point. This completes the derivation of the adjoint system in a weak form.
\end{proof}

For $v\in H^1(\Omega)$, we denote by $\langle v^\prime(\phi), \theta\rangle$ the Gâteaux derivative of $v$ at $\phi$ in the direction $\theta$. We have the following result on the Gâteaux derivative of the objective. 
\begin{Theorem}\label{sensitivityj}
Let $(\bm c,\psi)\in \left(H_{c,1}^1(\Omega)\times H_{c,2}^1(\Omega)\right)\times H^1_g(\Omega)$ and $(\bm s,\zeta)\in\textbf{H}^1_{d,0}(\Omega)\times H^1_{d,0}(\Omega)$ be the solutions of the steady-state PNP system (\ref{weakPNP}) and the adjoint system (\ref{adjweak}), respectively. Then the Gâteaux derivative of the objective $\mathcal{J}(\bm c(\phi))$ w.r.t. $\phi$ in the direction $\theta\in H^1(\Omega)$ reads
\begin{equation*}
\begin{aligned}
\big\langle\mathcal{J}^{\prime}(\bm c(\phi)), \theta\big\rangle=& -\sum_{i=1}^2\int_{\Omega} \langle D^\prime(\phi), \theta\rangle(\nabla c_i + z_i c_i \nabla \psi -\alpha_0 c_i \nabla \phi)\cdot \nabla s_i \dx-\int_{\Omega} \langle \epsilon^\prime(\phi),\theta\rangle\nabla\psi\cdot \nabla \zeta\dx\\
& +\sum_{i=1}^2\int_{\Omega}D(\phi)\alpha_0 c_i \nabla \theta\cdot \nabla s_i \dx.
\end{aligned}
\end{equation*}
\end{Theorem}
\begin{proof}
The Gâteaux derivative of the objective $\mathcal{J}(\bm c(\phi))$ w.r.t. the phase field function $\phi$ in the direction $\theta\in H^1(\Omega)$ yields
\begin{equation}\label{Jprime}
\begin{aligned}
\big\langle\mathcal{J}^{\prime}(\bm c(\phi)), \theta\big\rangle =  \int_{\Omega}j^\prime(x, \bm c(\phi))\langle \bm c^{\prime}(\phi),\theta\rangle \dx 
=  -\sum_{i=1}^2 \int_{\Omega} z_i \langle c_i^\prime(\phi), \theta\rangle \dx,
\end{aligned}
\end{equation}
where $j^\prime$ denotes the derivative w.r.t. the second variable. By utilizing the adjoint system \eqref{adjweak} and taking the test function $q = \langle c_i^\prime(\phi), \theta\rangle$ therein, we derive
\begin{equation}\label{s1}
\begin{aligned}
&\sum_{i=1}^2 \int_{\Omega} D(\phi) \nabla s_i \cdot \big(\nabla \langle c_i^\prime(\phi), {\theta}\rangle + z_i \langle c_i^\prime(\phi), {\theta}\rangle \nabla\psi-\alpha_0 \langle c_i^\prime(\phi), {\theta}\rangle\nabla\phi \big) \dx \\
&\quad= \sum_{i=1}^2\int_{\Omega} z_i\zeta \langle c_i^\prime(\phi), {\theta}\rangle-z_i \langle c_i^\prime(\phi), {\theta}\rangle \dx.
\end{aligned}
\end{equation}
Differentiating the Nernst--Planck equations in \eqref{weakPNP} w.r.t. \(\phi\) in the direction \(\theta \in H^1(\Omega)\) yields
\begin{equation}\label{s2}
\begin{aligned}
&\sum_{i=1}^2\int_{\Omega} [\langle D^\prime(\phi), {\theta}\rangle(\nabla c_i + z_i c_i \nabla \psi - \alpha_0 c_i \nabla \phi) +D(\phi)\big(\nabla \langle c_i^\prime(\phi), {\theta}\rangle + z_i \langle c_i^\prime(\phi), {\theta}\rangle \nabla \psi\big)]\cdot \nabla s_i \dx \\
&-\sum_{i=1}^2\int_\Omega D(\phi)\alpha_0\langle c_i^\prime(\phi), {\theta}\rangle\nabla \phi\cdot \nabla s_i + D(\phi)\alpha_0 c_i \nabla\theta\cdot\nabla s_i \dx\\
&+\sum_{i=1}^2\int_{\Omega}D(\phi)z_i c_i \nabla \langle \psi^\prime(\phi), {\theta}\rangle \cdot \nabla s_i \dx = 0,
\end{aligned}
\end{equation}
where $(s_1, s_2,\zeta)\in [{H}^1_{d,0}(\Omega)]^3$ solves the adjoint system \eqref{adjweak}. Similarly, differentiating the Poisson equation in \eqref{weakPNP} w.r.t. $\phi$ gives
\begin{equation}\label{s3}
\begin{aligned}
\int_{\Omega} \langle \epsilon^\prime(\phi),\theta\rangle\nabla\psi\cdot \nabla \zeta
+\epsilon(\phi)\nabla\langle\psi^\prime(\phi), \theta\rangle\cdot \nabla \zeta- \bigg(\sum_{i=1}^2 z_i \langle c_i^\prime(\phi), {\theta}\rangle\bigg)\zeta\dx= 0.
\end{aligned}
\end{equation}
Using \eqref{s1} and combining \eqref{Jprime}, the Gâteaux derivative can be rewritten as
\begin{equation*}
\begin{aligned}
\big\langle\mathcal{J}^{\prime}(\bm c(\phi)), \theta\big\rangle
=\sum_{i=1}^2 \int_{\Omega} D(\phi) \nabla s_i \cdot \big[\nabla \langle c_i^\prime(\phi), {\theta}\rangle +  \langle c_i^\prime(\phi), {\theta}\rangle (z_i\nabla\psi-  \alpha_0\nabla\phi) \big] -z_i\zeta \langle c_i^\prime(\phi), {\theta}\rangle \dx.
\end{aligned}
\end{equation*}
Furthermore, combining with \eqref{s2} and \eqref{s3}, we obtain
\begin{equation}\label{s4}
\begin{aligned}
\big\langle\mathcal{J}^{\prime}(\bm c(\phi)), \theta\big\rangle
=& -\sum_{i=1}^2\int_{\Omega} \langle D^\prime(\phi), \theta\rangle(\nabla c_i + z_i c_i \nabla \psi-\alpha_0 c_i \nabla \phi)\cdot \nabla s_i \dx \\
&+\sum_{i=1}^2\int_{\Omega}D(\phi)\alpha_0 c_i \nabla \theta\cdot \nabla s_i \dx -\sum_{i=1}^2\int_{\Omega}D(\phi)z_i c_i \nabla \langle \psi^\prime(\phi), \theta\rangle \cdot \nabla s_i \dx\\
&-\int_{\Omega} \langle \epsilon^\prime(\phi),\theta\rangle\nabla\psi\cdot \nabla \zeta+ \epsilon(\phi)\nabla\langle\psi^\prime(\phi), \theta\rangle\cdot \nabla \zeta \dx.
\end{aligned}
\end{equation}
Taking the test function \(\eta = \langle \psi^\prime(\phi), \theta \rangle\) for the adjoint system \eqref{adjweak} and substituting into \eqref{s4} completes the derivation.
\end{proof}
\subsection{Gradient flow minimization}
Given an energy functional bounded below, incorporating a volume constraint as a penalty term
\begin{equation}\label{modiEner}
\hat{\mathcal{W}}(\phi):=\mathcal{W}(\phi)+\frac{\beta}{2}(V(\phi)-V_0)^2,
\end{equation}
we denote the variational derivative by $\frac{\delta \hat{\mathcal{W}}}{\delta \phi}$, where $\beta>0$ is a prescribed parameter. 
To solve the optimal control problem $(\ref{TopOpt})$ and find a (local) minimum, we introduce an Allen-Cahn type of gradient flow \cite{JLXZ,Shen2010,JCP2010PF} 
\begin{equation}\label{GC}\left\{
\begin{aligned}
    &\frac{\partial \phi}{\partial t} = -\frac{\delta \hat{\mathcal{W}}}{\delta \phi},\quad {\rm in}\ \Omega, \\
    &\frac{\partial \phi}{\partial \bm n}=0,\quad\quad\quad  {\rm on}\ \Gamma_1; \ \ \phi=0,\quad {\rm on}\ \Gamma_2;\ \phi=1,\quad {\rm on}\ \Gamma_{\rm in}, \\
    &\phi( \cdot , 0) = \phi_0,\quad\, \,  {\rm in}\ \Omega,
    \end{aligned}\right.
\end{equation}
where $\phi_0:\Omega\rightarrow [0,1]$ is a prescribed initial phase field function. Also, the variational derivative of $\hat{\mathcal{W}}$ defined in \eqref{modiEner} is given by
\begin{equation*}
\frac{\delta \hat{\mathcal{W}}}{\delta \phi}=-\kappa \Delta\phi + \frac{1}{\kappa}\omega^\prime(\phi)+j^\prime(x,\bm c(\phi))+\beta (V(\phi)-V_0),
\end{equation*}
where $\omega^\prime(\phi)=\frac{1}{2}\phi(\phi-1)(2\phi-1)$. Here, the variational derivative of the objective denotes
\begin{equation*}
\begin{aligned}
 &\langle j^\prime(x,\bm c(\phi)), \theta\rangle \\
 &:=-\sum_{i=1}^2 \langle D^\prime(\phi), \theta\rangle(\nabla c_i + z_i c_i \nabla \psi-\alpha_0 c_i \nabla\phi)\cdot \nabla s_i - \langle\epsilon^\prime(\phi), \theta\rangle \nabla\psi\cdot \nabla \zeta+\sum_{i=1}^2D(\phi)\alpha_0 c_i \nabla \theta \cdot \nabla s_i,   
\end{aligned}
\end{equation*}
where $(c_1, c_2, \psi)\in H_{c,1}^1(\Omega)\times H_{c,2}^1(\Omega)\times H^1_g(\Omega)$ and $(s_1,s_2,\zeta)\in [H_{d,0}^1(\Omega)]^3$ are the solutions of \eqref{weakPNP} and \eqref{adjweak} for the given $\phi( x ,t)$, respectively.

\section{Numerical methods}

\subsection{Finite-element discretization of state and adjoint}
We consider a family of quasi-uniform triangulations/tetrahedralizations $\{ \mathcal{T}_h\}_{h >0}$ of $\Omega$, where each $\mathcal{T}_h$ partitions polyhedral $\Omega$ into triangles (2D) and tetrahedrons (3D) such that $\overline{\Omega} = \bigcup _{K \in\mathcal{T} _h} \overline{K}$. The mesh size is defined as $ h:= \max_{K\in\mathcal{T}_h} h_K$, where $h_K$ denotes the diameter of any $K \in \mathcal{T}_h$. 
We introduce the finite-dimensional subspace $W_h=\{q_h\in H^1(\Omega)\cap C^0(\overline{\Omega}): \,q_h|_K\in \mathbb{P}_1\ \forall K\in\mathcal{T}_h\}$, consisting of continuous piecewise linear polynomials on $\mathcal{T}_h$. 

Denote $W_{g,h}:=\{q_h\in W_h: q_h = g\ {\rm on}\ \Gamma_{\rm in}\cup\Gamma_2\}$ for discrete potential $\psi_h$, $\textbf{W}_h:= [W_h]^2$, and $\textbf{W}_{\bm \rho,h}:=\{\bm \rho_h\in \textbf{W}_h :  \rho_{i,h} =  \rho_i^\infty\ {\rm on}\ \Gamma_{\rm in}\cup \Gamma_2,\ i=1,2 \}$ for Slotboom variables where $\bm \rho_h = [\rho_{1,h},\rho_{2,h}]^{\rm T}$. Introduce $W_{0,h}:=\{ w_h\in W_h: w_h=0\ {\rm on}\ \Gamma_{\rm in}\cup\Gamma_{2} \}$ and denote $\textbf{W}_{0,h}:=[W_{0,h}]^2$. 
Then the discrete PNP system reads: Find $(\bm \rho_h,\psi_h)\in \textbf{W}_{\bm \rho,h} \times W_{g,h}$ such that for $i=1,2$
\begin{equation}\label{PNPdis}\left\{
\begin{aligned}
&\int_{\Omega} D(\phi_h) \exp{(\alpha_0\phi_h-z_i\psi_h)}\nabla \rho_{i,h}\cdot \nabla v_{i,h} \dx =0 \quad&& \forall \bm v_h\in \textbf{W}_{0,h}, \\
&\int_{\Omega} \epsilon(\phi_h) \nabla \psi_h \cdot \nabla \zeta_h \dx = \int_{\Omega} \bigg(\sum_{i=1}^2 z_i \rho_{i,h} \exp{(\alpha_0\phi_h- z_i \psi_h)}\bigg) \zeta_h \dx \quad && \forall \zeta_h\in W_{0,h},
\end{aligned}\right.
\end{equation}
where $\bm v_h=[v_{1,h}, v_{2,h}]^{\rm T}$ is a vectorial test function. To solve the nonlinear system \eqref{PNPdis} efficiently, we use the well-known Gummel fixed-point scheme \cite{Gummel1964} that features good convergence properties with respect to the initial guess. Its convergence property in continuous space has been discussed in \cite[page 332]{Markowich1986}.  The fixed-point iterative scheme alternatively solves the Poisson--Boltzmann-type equation for the electric potential and a self-adjoint continuity system for the ionic concentrations. The Poisson--Boltzmann-type equation can be efficiently solved using the Newton’s iterations that can converge locally at a quadratic rate.

In electric double layers next to the electrode, sharp boundary layers could develop due to the relatively small Debye length in electroneutral scenarios. Singular coefficients $\exp{(\alpha_0\phi_h- z_i\psi_h)}$ in \eqref{PNPdis} can cause severe numerical instability due to their sharp variations between elements. To stabilize, the discrete form of the continuity equation in \eqref{PNPdis} is approximated by the inverse-average techniques~\cite{Brezzi1989,averageJCP2022} as
\begin{equation}\label{InvAve}
\begin{aligned}
&\sum_{K\in\mathcal{T}_h}\int_K D(\phi_h) \exp{(\alpha_0 \phi_h-z_i\psi_h)}\nabla \rho_{i,h}\cdot \nabla v_{i,h} \dx\\
&\quad\approx \sum_{K\in\mathcal{T}_h} E(\alpha_0 \phi_h-z_i\psi_h)_K \int_K D(\phi_h) \nabla \rho_{i,h}\cdot \nabla v_{i,h} \dx,
\end{aligned}
\end{equation}
and
\begin{equation*}
\begin{aligned}
\sum_{K\in\mathcal{T}_h}\int_{K} \bigg(\sum_{i=1}^2 z_i \rho_{i,h} \exp{(\alpha_0\phi_h- z_i \psi_h)}\bigg) \zeta_h \dx
\approx\sum_{K\in\mathcal{T}_h}\sum_{i=1}^2E(\alpha_0\phi_h-z_i\psi_h)_K\int_{K}  z_i \rho_{i,h}  \zeta_h \dx
\end{aligned}
\end{equation*}
where 
\begin{equation*}
E(\alpha_0\phi_h-z_i\psi_h)_K = \bigg(\frac{1}{\vert K \vert}\int_K \exp{(-\alpha_0\phi_h+z_i\psi_h)}\dx \bigg)^{-1}.
\end{equation*}
If the mesh $\mathcal{T}_h$ satisfies certain quality conditions, the matrix associated with (\ref{InvAve}) is an M-matrix \cite{Brezzi1989}, which brings good performance on  mass conservation and non-negativity. Interested readers are referred to our previous work~\cite{LZZIJNME2025} for more details.

Moreover, the finite element discretization of the adjoint problem \eqref{adjweak} reads: Find $\bm s_h=[s_{1,h},s_{2,h}]^{\rm T} \in \textbf{W}_{0,h}$ and $\zeta_h \in W_{0,h}$ such that for $i=1,2$
\begin{equation}\label{Adjdis}\left\{
\begin{aligned}
&\int_{\Omega} D(\phi_h) \nabla s_{i,h}\cdot(\nabla v_{i,h}+z_i v_{i,h}\nabla \psi_h- \alpha_0 v_{i,h} \nabla \phi_h) + z_i v_{i,h} -z_i\zeta_h v_{i,h} \dx =0, \quad \forall \bm v_h\in \textbf{W}_{0,h}, \\
&\int_{\Omega} \epsilon(\phi_h) \nabla \zeta_h \cdot \nabla \xi_h \dx+\sum_{i=1}^2\int_{\Omega} D(\phi_h) z_i c_{i,h}\nabla s_{i,h}\cdot \nabla \xi_h \dx=0,\quad \forall \xi_h\in W_{0,h}.
\end{aligned}\right.
\end{equation}
\subsection{Optimization algorithm}
For a terminal time $T>0$, we simply consider a uniform time partition $0=t_0 < t_1<\cdots < t_n<t_{n+1}<\cdots< t_{\tilde{N}}=T$ with a time step size \(\nu := T/\tilde{N}\) for \(\tilde{N} \in \mathbb{N}\). Denote by semi-discrete approximations $\phi^{n+1}\approx \phi( x , \nu_{n+1})$ and $\phi^n \approx \phi( x , \nu_n)$. To solve the topology optimization problem (\ref{TopOpt}), we propose a stabilized semi-implicit time discretization scheme
\begin{equation}\label{ACstabScheme}
\begin{aligned}
&\frac{\phi^{n+1}-\phi^n}{\nu} - \kappa \Delta \phi^{n+1}+\Lambda_1(\phi^{n+1} - \phi^n)-\Lambda_2 \Delta(\phi^{n+1}- \phi^n) \\
&\quad=-\frac{1}{2\kappa}\phi^n(\phi^n-1)(2\phi^n-1) -\beta \big(V(\phi^n) - V_0\big) +j^\prime(x,\bm c^n),
\end{aligned}
\end{equation}
where $(c_1^n,c_2^n,\psi^n):=(c_1(\phi^n),c_2(\phi^n),\psi(\phi^n))\in H_{c,1}^1(\Omega)\times H_{c,2}^1(\Omega)\times H^1_g(\Omega)$ are the solutions of \eqref{weakPNP} for the prescribed $\phi^n\in H^1(\Omega), 0\leq\phi^n \leq 1\ {\rm a.e.}\ {\rm in}\ \Omega$. To achieve efficient time stepping with the cost of a linear solver at each time step, the nonlinear terms are treated explicitly and stabilization terms with parameters \(\Lambda_1, \Lambda_2 > 0\) are introduced to facilitate stability; cf.~\cite{Shen2010,Shen2019}. 

The phase field function may exceed the range $[0,1]$. As a remedy, we use a projected gradient descent method (see, e.g., \cite{LiFu} therein). After obtaining \(\phi^{n+1}\) via \eqref{ACstabScheme}, we project it onto \([0,1]\):
$\phi^{n+1}( x )\leftarrow \mathcal{P}(\phi^{n+1}( x ))$,
where $\mathcal{P}(\phi( x )):=\min( \max(\phi( x ),0),1)$.

We present a topology optimization algorithm based on the gradient flow scheme \eqref{ACstabScheme}. To save computational cost, the PNP system \eqref{PNPdis} and adjoint equations \eqref{Adjdis} are updated once after the gradient flow scheme advances 10 time steps. The computational details is summarized in Algorithm~\ref{alg1}.

\begin{algorithm}[htb]
\SetAlgoLined
\caption{Topology optimization of supercapacitor electrode.
}\label{alg1}
\KwData{Given outer iteration number $N$, and target volume \(V_0\);\\
\qquad \, \,  Initialize phase field function $\phi_0$ and set $n=0$}
\While {$n \leq N$}{
\emph{Step 1}: Solve the PNP system \eqref{PNPdis} to obtain potential and concentrations \\
\emph{Step 2}: Solve the adjoint problem \eqref{Adjdis} to obtain adjoint variables \\
\emph{Step 3}: Update the phase field function via gradient flow scheme \eqref{ACstabScheme}\\
$n \leftarrow n+1$\\
}
\end{algorithm}

\section{Numerical examples}
Numerical simulations are performed using FreeFem++ \cite{Hecht} on a personal computer equipped with a 12th Gen Intel(R) Core(TM) i7-12700 (2.10 GHz) and 16.0 GB. Unless specified otherwise, we take the following parameters in our numerical simulations:
$\epsilon_0 = 0.01$, $\epsilon_m = 5$, $D_0 = 0.5$, $D_m = 0.01$, $g\big\vert_{\Gamma_2} = -0.5$, and $c_1^{\infty}\vert_{\Gamma_{\rm in}} = c_2^{\infty}\vert_{\Gamma_{\rm in}} = 0.5, L=1, p =2, \alpha_0=1$.


\subsection{Maximize total charge storage}
\begin{figure}[htbp]
\centering
\includegraphics[width=3.5 in]{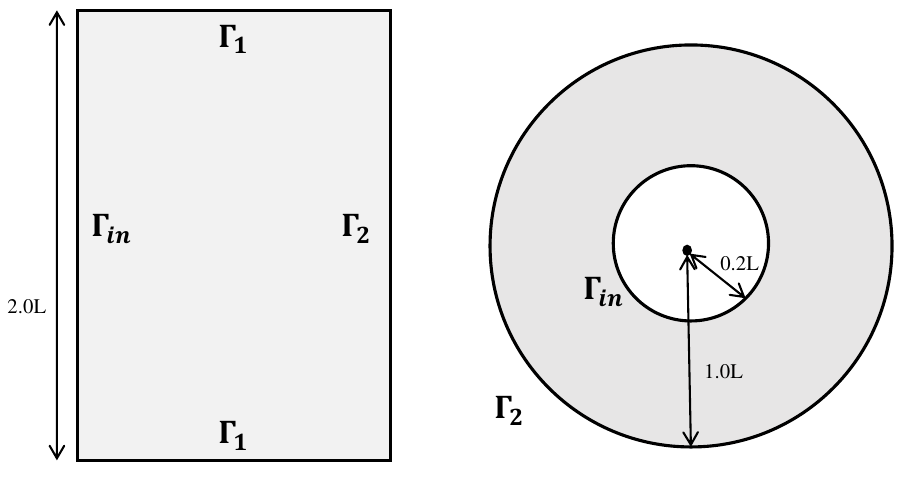}
\caption{A rectangular design domain for Example 1 (left) and an annulus design domain for Example 2 (right).}
\label{shapeOpt2D} 
\end{figure}
The initial phase field function in 2D is given by \(\phi_0 = 0.5 + 0.5\cos(m\pi x_1)\cos(m\pi x_2)\) for Example 1 and Example 2 below, where $m$ is a positive integer.

\begin{figure}[htbp]
\begin{minipage}[b]{0.166\textwidth}
\centering
    \includegraphics[width=0.87 in]{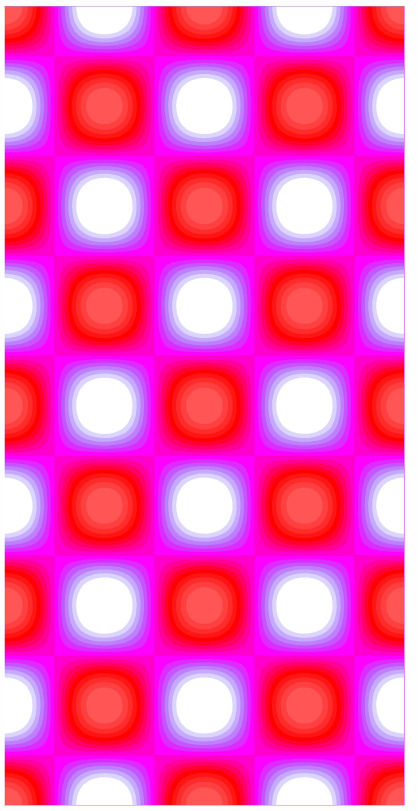}
\end{minipage}
\begin{minipage}[b]{0.166\textwidth}
\centering
    \includegraphics[width=0.87 in]{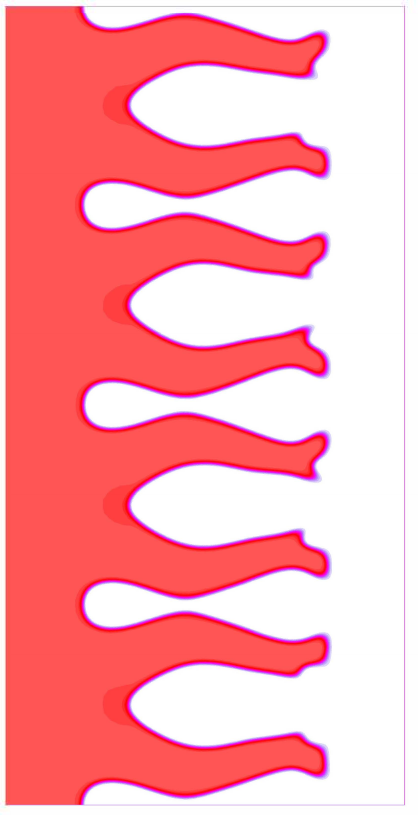}
\end{minipage}
\begin{minipage}[b]{0.166\textwidth}
\centering
    \includegraphics[width=0.87 in]{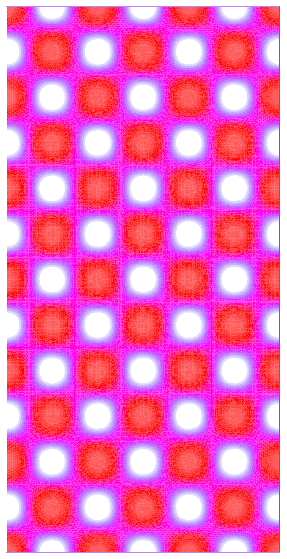}
\end{minipage}
\begin{minipage}[b]{0.166\textwidth}
\centering
    \includegraphics[width=0.9 in]{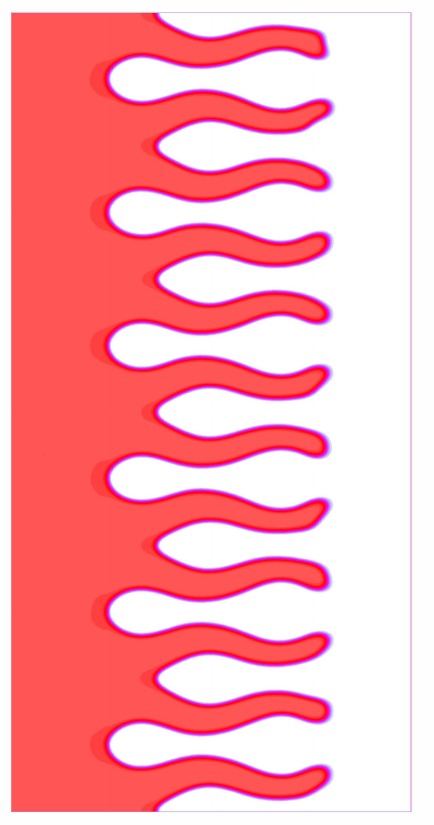}
\end{minipage}
\begin{minipage}[b]{0.166\textwidth}
\centering
    \includegraphics[width=0.87 in]{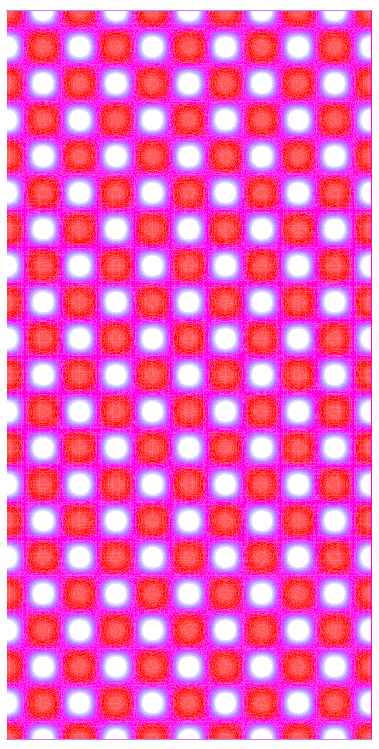}
\end{minipage}
\begin{minipage}[b]{0.166\textwidth}
\centering
    \includegraphics[width=0.87 in]{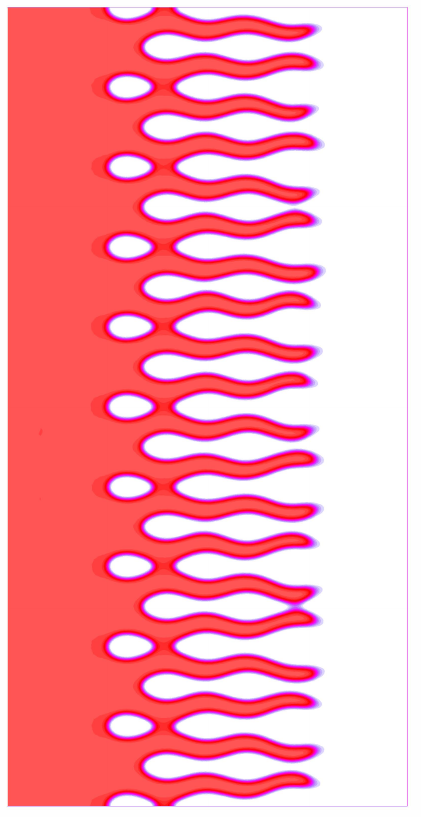}
\end{minipage}
\caption{Initial and optimal design for Example 2 with \(m = 4, 6, 10\) from left to right.}
\label{exp1cas2} 
\end{figure}

\textbf{Example 1}: Consider a design domain \(\Omega = (0, 1) \times (0, 2)\) as shown in the left panel of Fig.~\ref{shapeOpt2D}. Parameters are set as: \( \nu = 2\times 10^{-4}\), \(V_0 = 1\), \(\kappa = 10^{-3}\), \(\Lambda_1 = 1\), \(\Lambda_2 = 10^{-2}\), and \(\beta = 500\).

\begin{figure}[htbp]
\begin{minipage}[b]{0.33\textwidth}
\centering
    \includegraphics[width=1.4 in]{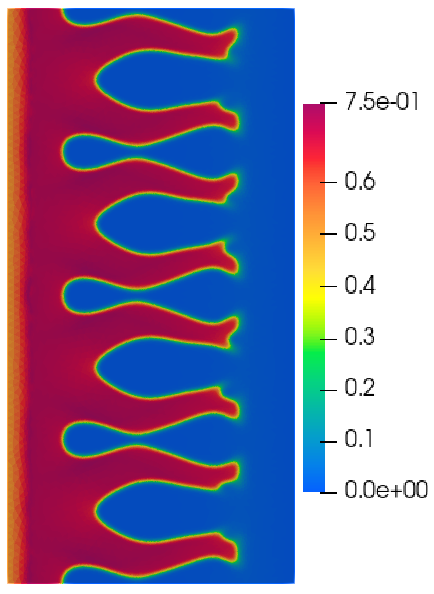}
\end{minipage}
\begin{minipage}[b]{0.33\textwidth}
\centering
    \includegraphics[width=1.475 in]{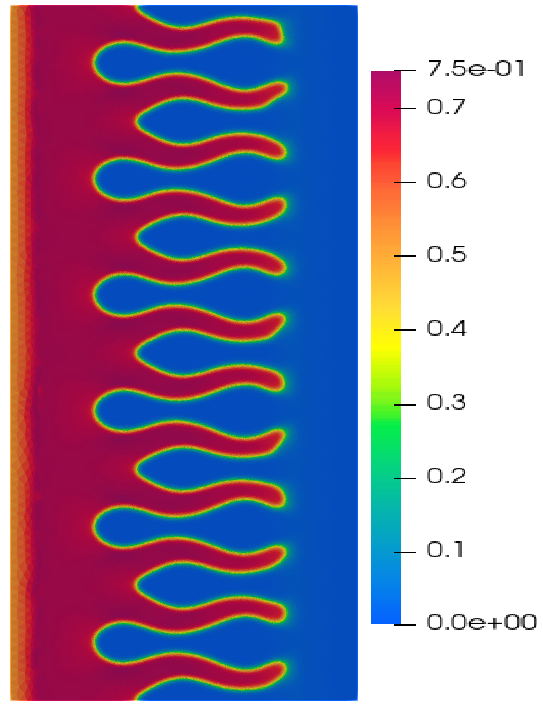}
\end{minipage}
\begin{minipage}[b]{0.33\textwidth}
\centering
    \includegraphics[width=1.475 in]{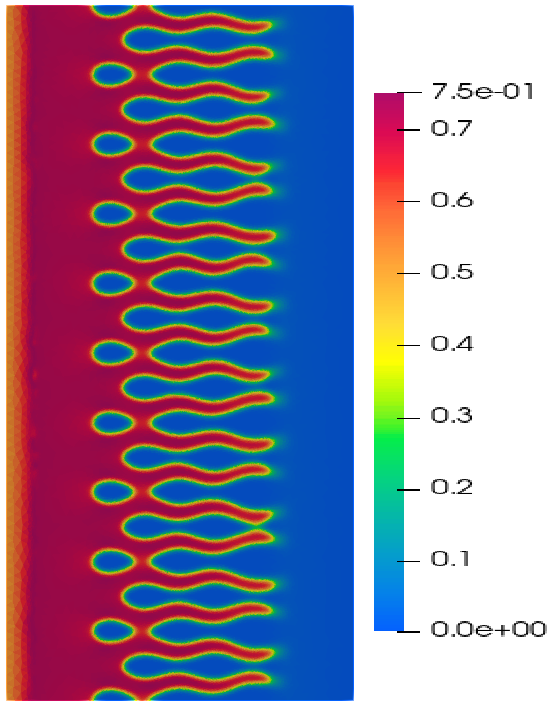}
\end{minipage}
\\
\begin{minipage}[b]{0.33\textwidth}
\centering
    \includegraphics[width=1.4 in]{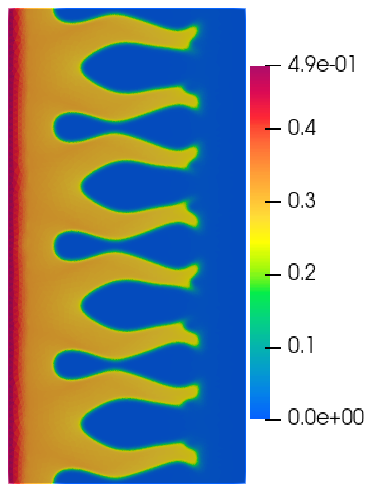}
\end{minipage}
\begin{minipage}[b]{0.33\textwidth}
\centering
    \includegraphics[width=1.4 in]{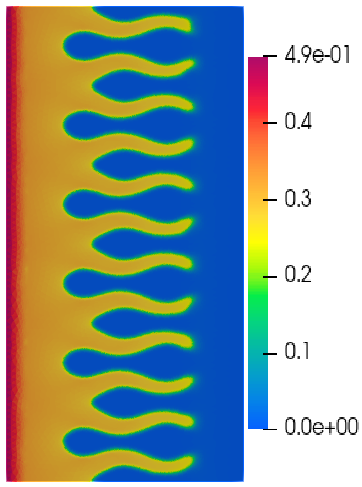}
\end{minipage}
\begin{minipage}[b]{0.33\textwidth}
\centering
    \includegraphics[width=1.45 in]{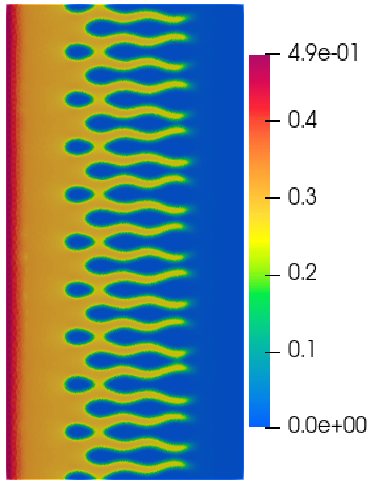}
\end{minipage}
\caption{Concentrations of cations (upper row) and anions (lower row) for optimized configurations in Example 1 ($m=4, 6, 10$ from left to right).}
\label{exp1c1cas2} 
\end{figure}

\begin{figure}[htbp]
\begin{minipage}[b]{0.5\textwidth}
\centering
    \includegraphics[width=2.3 in]{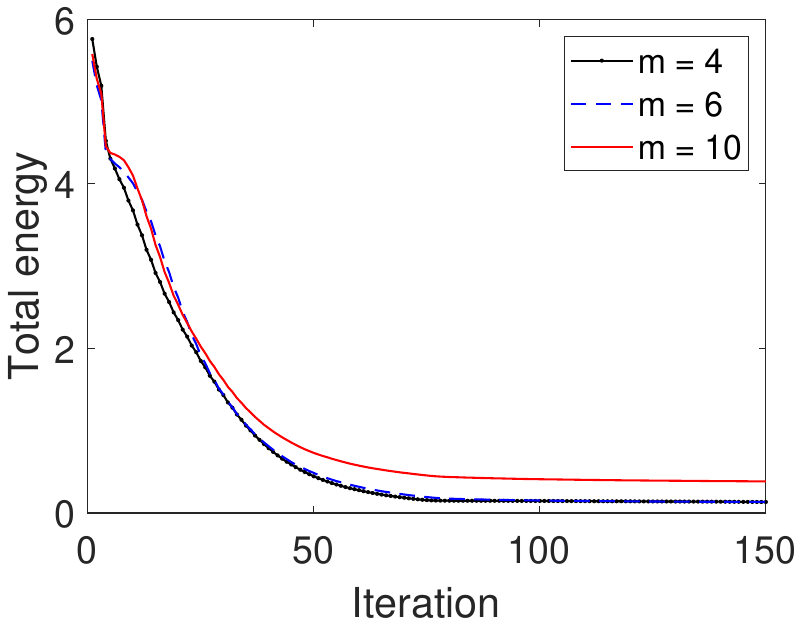}
\end{minipage}
\begin{minipage}[b]{0.5\textwidth}
\centering
    \includegraphics[width=2.3 in]{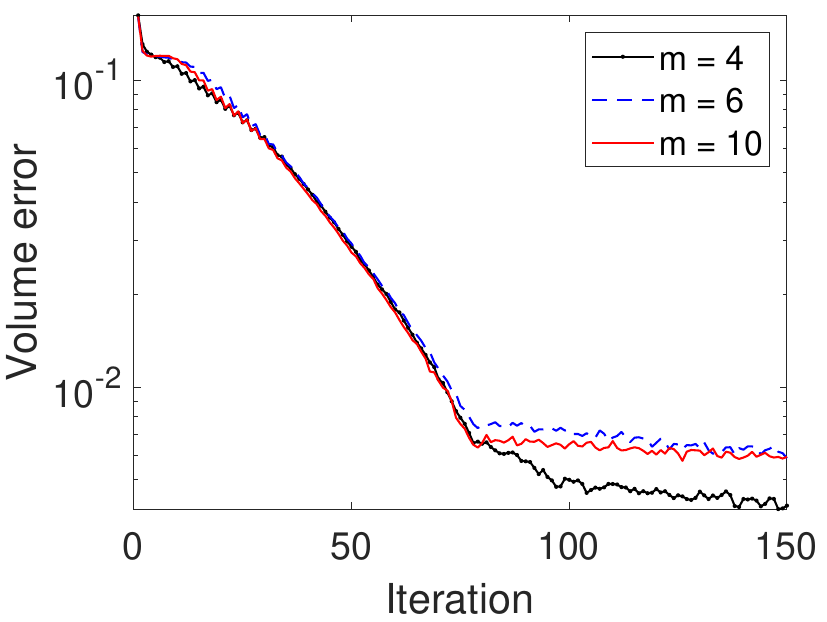}
\end{minipage}
\caption{Convergence histories of energy (left) and volume error (right) for Example 1.}
\label{exp1case2BObjVol} 
\end{figure}

The initial and optimal configurations for different $m=4, 6, 10$ are shown in Fig.~\ref{exp1cas2}, which indicates that more ``holes" in the initial phase field lead to more complex, delicate final structures with larger interfacial area. Fig. \ref{exp1c1cas2} displays concentration distributions of both ion species in optimal configurations. Clearly, the anions are repelled from the negatively charged electrode and cations are attracted into the porous-like structure.  As positive ions accumulate near the interface, wavy structures with larger contact areas emerge to maximize ion storage. Fig. \ref{exp1case2BObjVol} shows the convergence histories of the objective and total energy with a volume error lower than $0.01$. It can be found that the supercapacitor storage achieves roughly 4 times as much as the initial configuration after optimization.  This is consistent with the well-known fact that larger supercapacitor storage can be achieved if larger interfacial area is designed, since energy is mainly stored in the electric double layers at the interface.

\begin{figure}[htbp]
\begin{minipage}[b]{0.33\textwidth}
    \includegraphics[width=1.6 in]{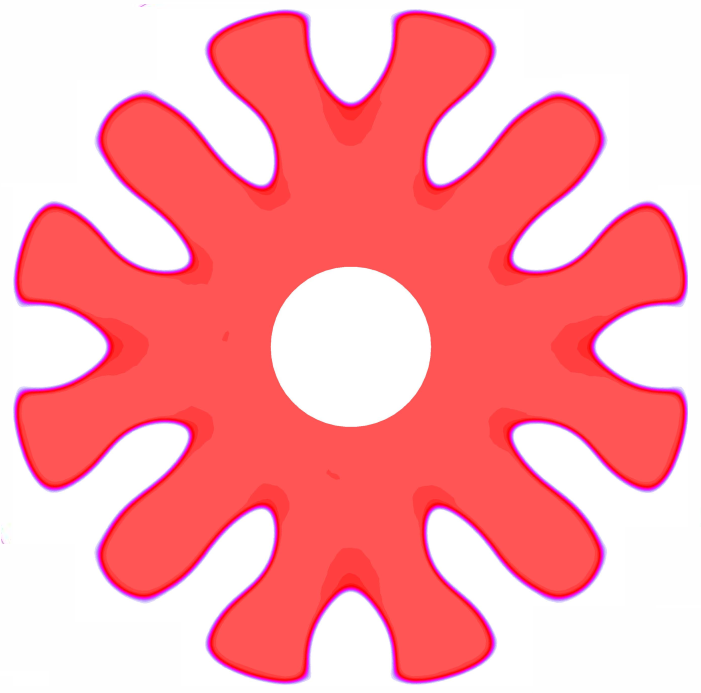}
\end{minipage}
\begin{minipage}[b]{0.33\textwidth}
    \includegraphics[width=1.6 in]{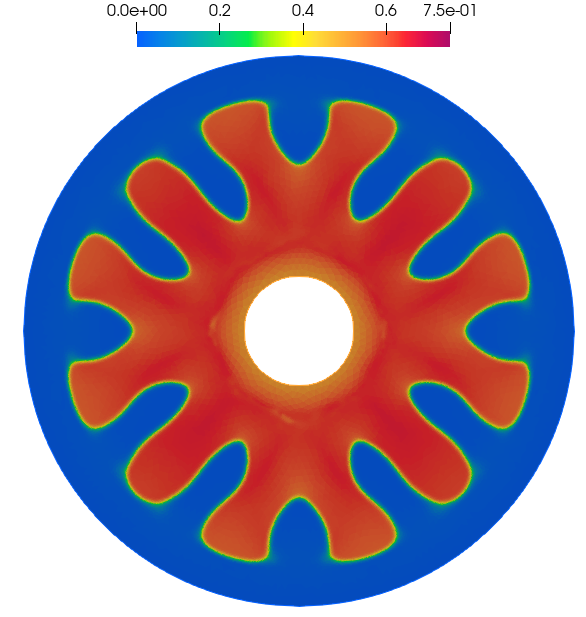}
\end{minipage}
\begin{minipage}[b]{0.33\textwidth}
    \includegraphics[width=1.6 in]{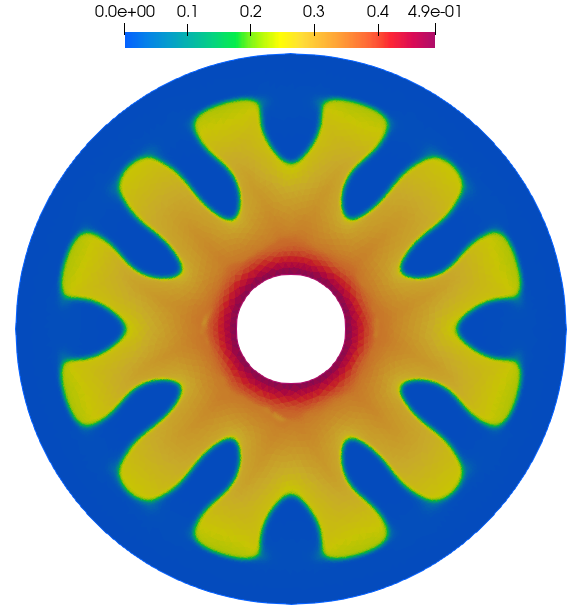}
\end{minipage}
\\
\begin{minipage}[b]{0.33\textwidth}
    \includegraphics[width=1.6 in]{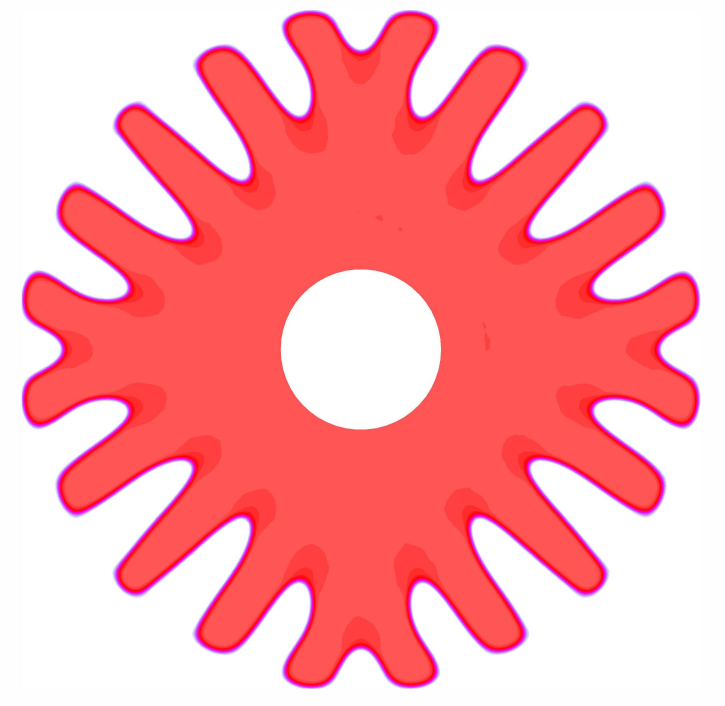}
\end{minipage}
\begin{minipage}[b]{0.33\textwidth}
    \includegraphics[width=1.6 in]{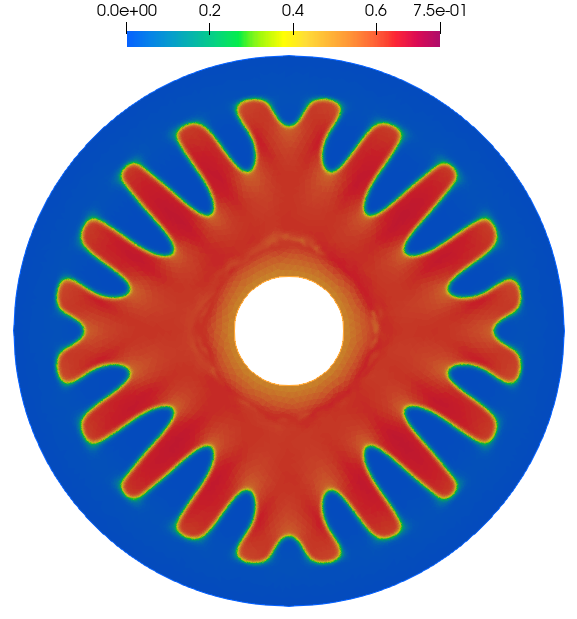}
\end{minipage}
\begin{minipage}[b]{0.33\textwidth}
    \includegraphics[width=1.6 in]{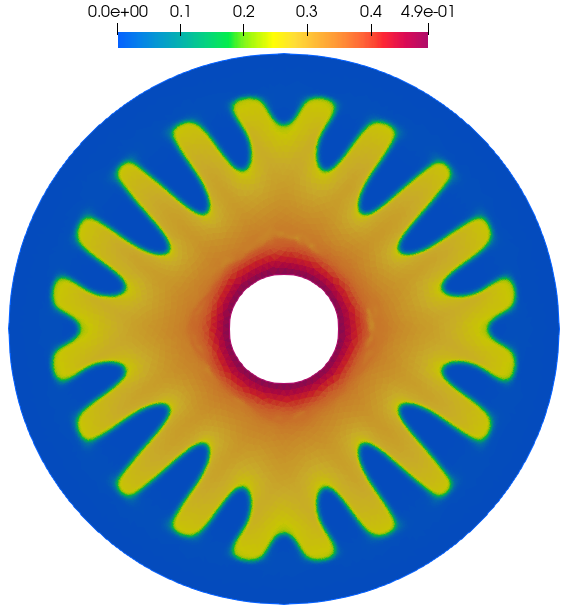}
\end{minipage}
\caption{The optimal configuration (column 1), and the corresponding concentrations of cations (column 2) and anions (column 3) with different initial phase field functions ($m=4, 6$ from upper row to lower row) in Example 2.}
\label{exp2Opt} 
\end{figure}

\begin{figure}[htbp]
\begin{minipage}[b]{0.5\textwidth}
\centering
    \includegraphics[width=2.5 in]{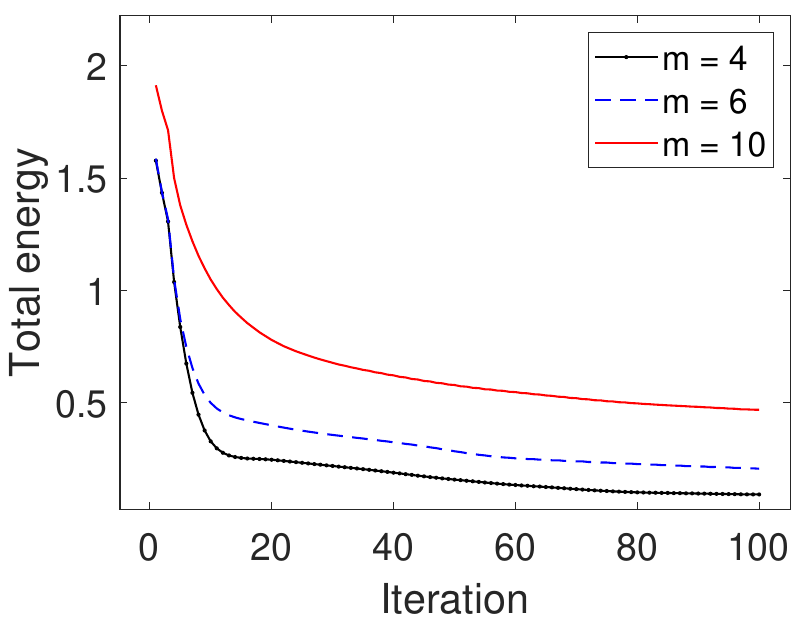}
\end{minipage}
\begin{minipage}[b]{0.5\textwidth}
\centering
    \includegraphics[width=2.5 in]{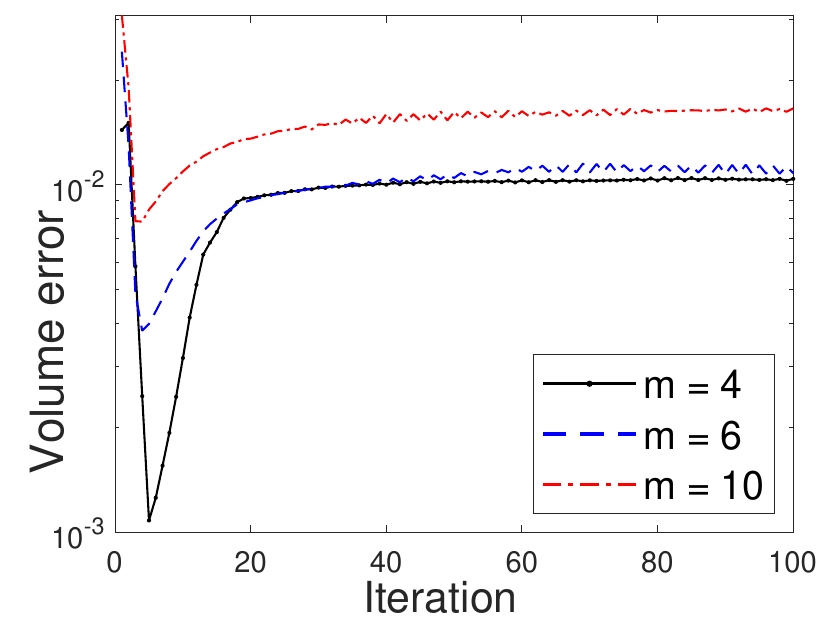}
\end{minipage}
\caption{The convergence histories of total energy (left) and volume error (right) for Example 2.}
\label{exp2ObjVol} 
\end{figure}

\textbf{Example 2}: In this example, we consider topology optimization of cylindrical supercapacitors~\cite{Mei2017,Simon2008}.  As shown in Fig.~\ref{shapeOpt2D}, the design domain $\Omega:=\{(x_1,x_2)\vert 0.04< x_1^2+x_2^2<1 \}$ is given by an annulus centered at the origin \((0,0)\), with the inner and outer ring radii being \(0.2\) and \(1\), respectively.  The electrode being placed at the outer ring. Simulations are performed with the parameters: \(\nu = 10^{-3}\), \(\kappa = 10^{-3}\), \(\Lambda_1 = 1\), \(\Lambda_2 = 10^{-2}\), \(\beta = 500\), \(V_0 = 0.5|\Omega|\).

The first column of Fig.~\ref{exp2Opt} presents the optimized configurations with petal-shaped structures for initial phase field functions with \(m = 4, 6\). Again, our topology optimization gives optimized structures that favor enhancing the electrode-electrolyte interface area.  Columns 2 and 3 of Fig. \ref{exp2Opt} display concentration distributions of cations and anions corresponding to the optimized configuration, respectively. Such plots again demonstrate that the cations penetrate into tips and form electric double layers to store energy.  Convergence histories of total energy and volume errors shown in Fig.~\ref{exp2ObjVol} demonstrate the gradient flow scheme converge remains effective in multiple connected domains and gives convergence with a small volume error.

\begin{figure}[htbp]
\begin{minipage}[b]{0.33\textwidth}
\centering
    \includegraphics[width=1.45 in]{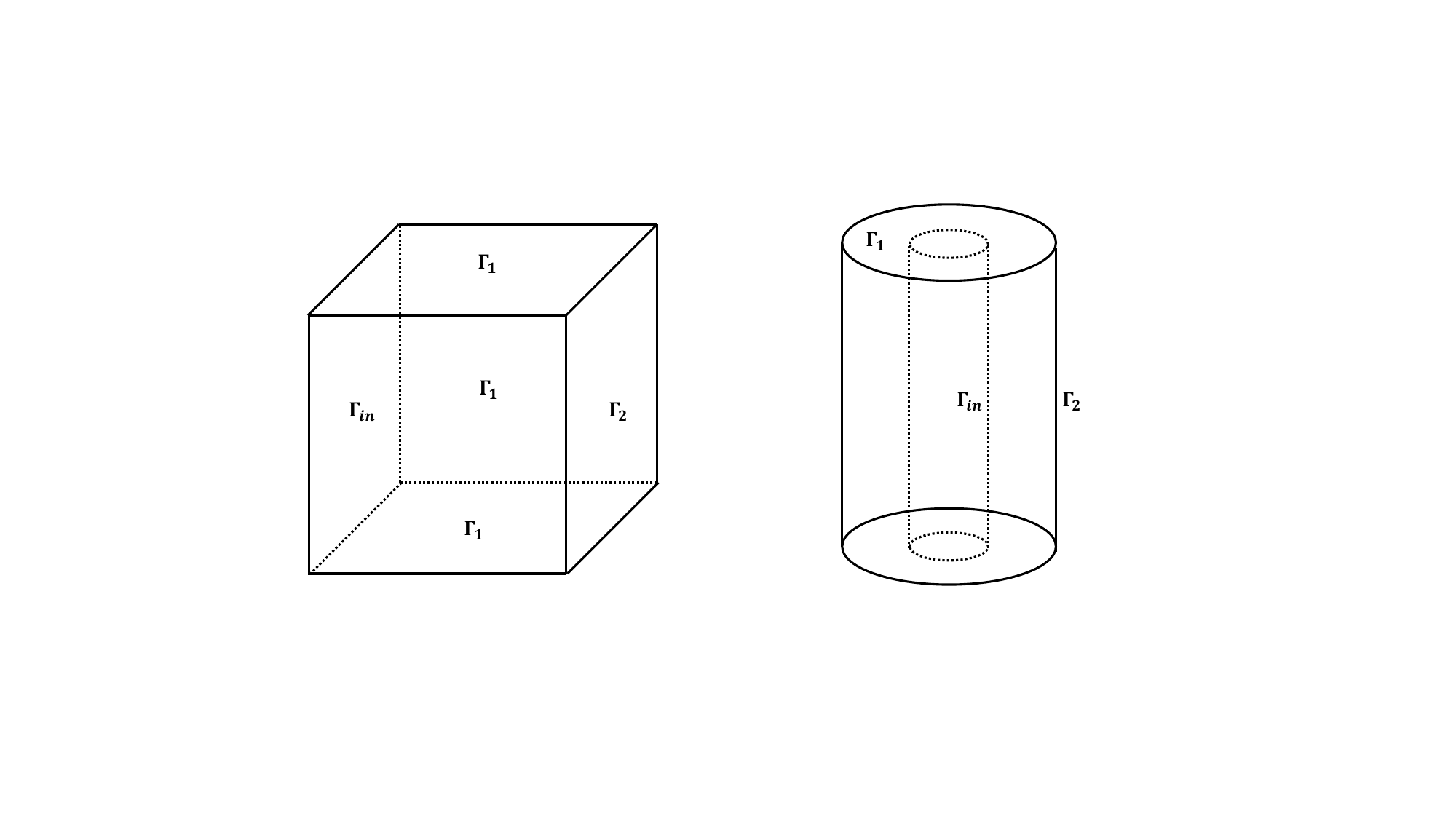}
\end{minipage}
\begin{minipage}[b]{0.33\textwidth}
\centering
    \includegraphics[width=1.3 in]{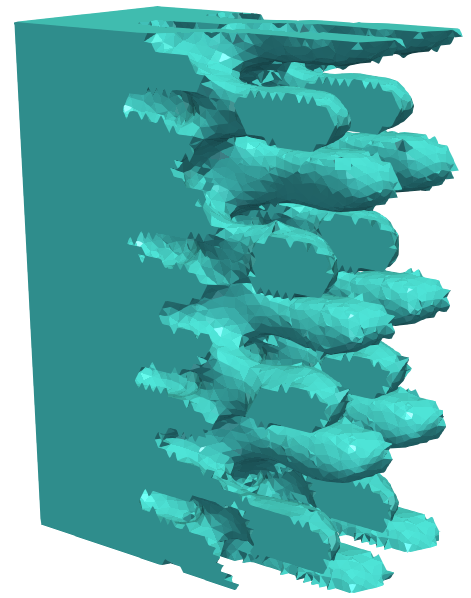}
\end{minipage}
\begin{minipage}[b]{0.33\textwidth}
\centering
    \includegraphics[width=1.9 in]{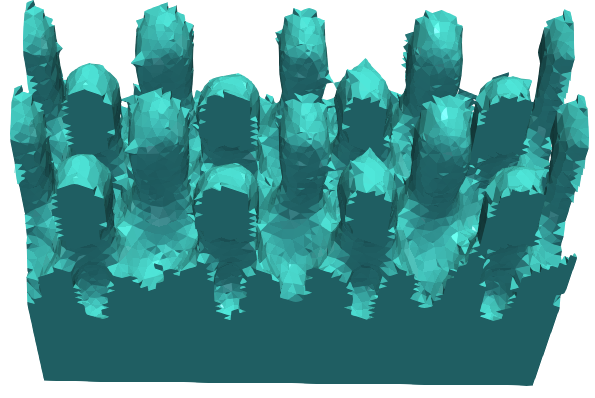}
\end{minipage}
\caption{Design domain (left), and optimized electrolyte regions (middle and right) for Case 1 in Example 3.}
\label{Exp3OptShape} 
\end{figure}

\begin{figure}[htbp]
\begin{minipage}[b]{0.33\textwidth}
\centering
    \includegraphics[width=1.2 in]{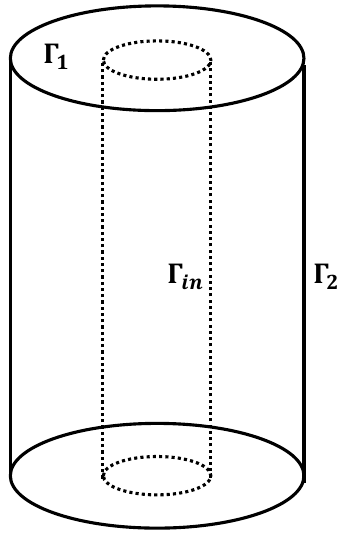}
\end{minipage}
\begin{minipage}[b]{0.33\textwidth}
\centering
    \includegraphics[width=1.2 in]{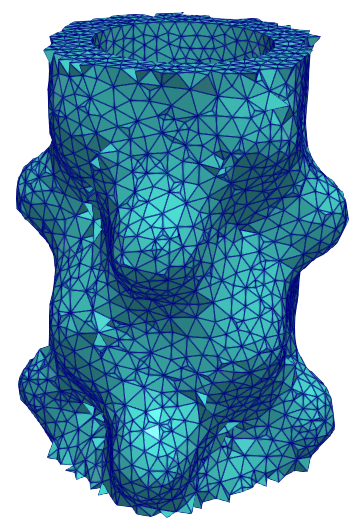}
\end{minipage}
\begin{minipage}[b]{0.34\textwidth}
\centering
    \includegraphics[width=1.8 in]{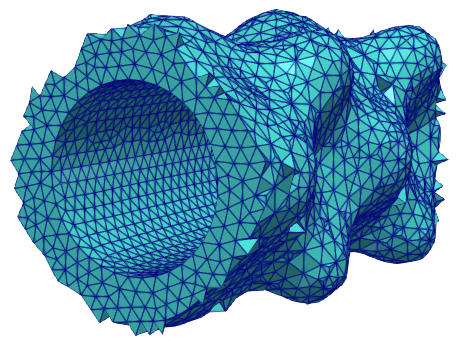}
\end{minipage}
\caption{Design domain (left), and optimized electrolyte regions (middle and right) for Case 2 in Example 3.}
\label{Exp3OptShapecase2} 
\end{figure}

\begin{figure}[htbp]
\begin{minipage}[b]{0.5\textwidth}
\centering
    \includegraphics[width=2.5 in]{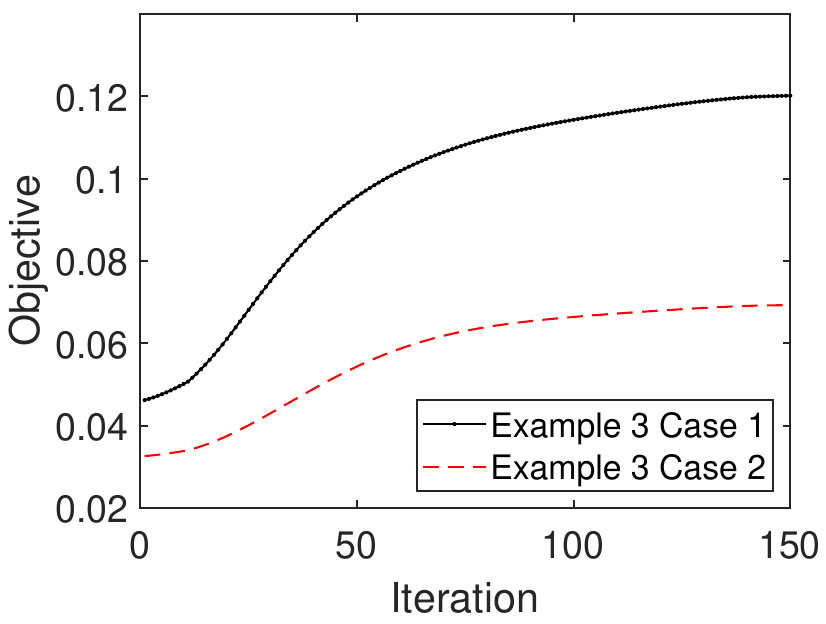}
\end{minipage}
\begin{minipage}[b]{0.5\textwidth}
\centering
    \includegraphics[width=2.5 in]{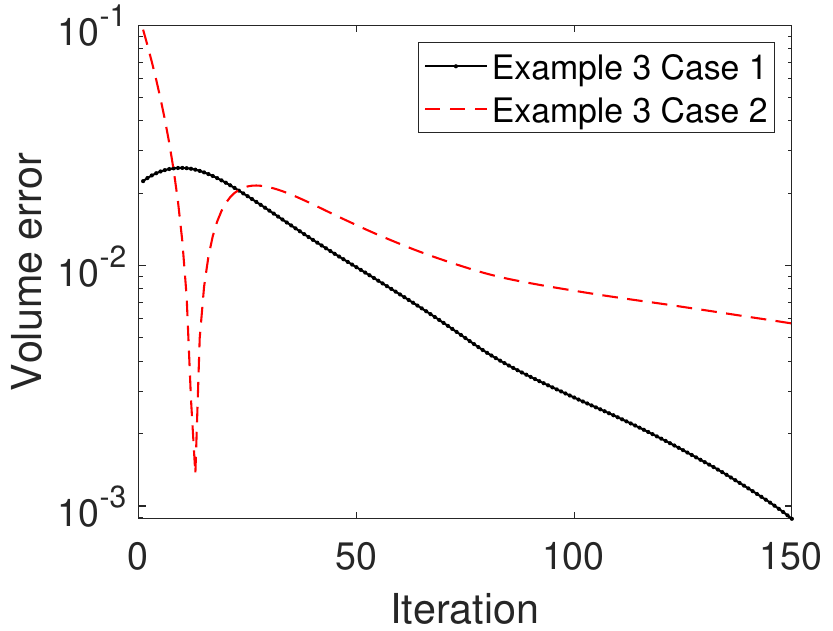}
\end{minipage}
\caption{Convergence histories of the objective (left) and volume error (right) for Example 3.}
\label{exp3BObjEnergycase2} 
\end{figure}

\textbf{Example 3}: In this example, we consider topology optimization of maximizing the charge storage of supercapacitors in 3D. The following two cases consider the cubic and cylindrical shapes of electrode structures. 

\emph{Case 1}: Consider a design domain \(\Omega = (0,1) \times (0,1) \times (0,0.4)\), as shown in the left panel of Fig.~\ref{Exp3OptShape}. Denote by \(\Gamma_{\rm in}\) the left face and \(\Gamma_2\) the right face of the cuboid. The computational domain is discretized by $169,011$ tetrahedral elements. The parameters of the PNP system \eqref{PNPdis} are chosen the same as in the Example 2. For the gradient flow, we set \(\nu = 10^{-4}\), \(\kappa = 10^{-3}\), \(V_0 = 0.6|\Omega|\), \(\Lambda_1 = 1\), \(\Lambda_2 = 0.5\), \(\beta = 2\times 10^{3}\). The initial phase field function is given by \(\phi_0(x_1, x_2, x_3) = 0.5 + 0.5\cos(4\pi x_1)\cos(8\pi x_2)\cos(10\pi x_3).\) The optimal electrolyte region computed by the Algorithm \ref{alg1} is visualized in Fig.~\ref{Exp3OptShape} from two different view angles. A porous structure is observed with exceptional capacity to store net charges. The convergence histories of the objective, as depicted in Fig. \ref{exp3BObjEnergycase2}, demonstrate that our topology optimization algorithm remains effective in 3D design of the electrode structure, with more than $15$ times enhancement  in charge storage. 

\emph{Case 2}: The design domain is set as a cylindrical annulus \(\Omega = \Omega_c \times [0,1]\), where \(\Omega_c = \{(x_1,x_2) \mid x_1^2 + x_2^2 \geq 0.04,\ x_1^2 + x_2^2 \leq 0.25\}\); cf.~ the left panel of Fig. \ref{Exp3OptShapecase2}. The computational mesh comprises $74,604$ tetrahedral elements and $14,837$ vertices. The parameters of the PNP system and phase field model remain the same as in the Case 1, except for \(V_0 = 0.3|\Omega|\) and the initial phase field function \(\phi_0(x_1, x_2, x_3) = 0.5 + 0.5\cos(4\pi x_1)\cos(4\pi x_2)\cos(4\pi x_3).\) The optimal electrolyte region is presented from two different view angles in Fig. \ref{Exp3OptShapecase2}. Again, a wavy eletrode-eletrolyte interface can be found to maximize the total charge storage. As shown in Fig. \ref{exp3BObjEnergycase2},  the objective converges robustly with a final volume error below $0.01$, further validating the effectiveness of proposed algorithm for storage maximization of supercapacitor in 3D domains. 



\section{Conclusion}
This work has proposed a topology optimization model based on the phase field method for maximizing charge storage in supercapacitors.  The model has been expressed in the form of a optimal control problem constrained by a modified steady-state Poisson--Nernst--Planck system, which is newly developed to describe ionic electrodiffusion in electrolyte-accessible domain during topology optimization.  The existence of minimizers to the optimal control problem has been theoretically established by using the direct method in the calculus of variation. Sensitivity analysis has been conducted to derive variational derivatives and corresponding adjoint equations. A gradient flow formulation, implemented with a stabilized semi-implicit scheme, has been developed to solve the resulting topology optimization problem. Ample 2D/3D numerical experiments have been performed to explore the optimal design of electrode structure. Various porous configurations have been found to have  high charge storage, validating the proposed topology optimization model and corresponding algorithm.
 


\end{document}